\documentclass[11pt,reqno]{amsart}
\usepackage[utf8]{inputenc}
\usepackage{mathtools,graphicx,xcolor,float,minipage-marginpar,blindtext,amssymb,amsmath,amsthm,cleveref,enumitem}
\graphicspath{{image/}}
\usepackage{wrapfig}
\usepackage[mathscr]{euscript}
\usepackage[a4paper, hmargin=2cm, vmargin = 3cm]{geometry}

\allowdisplaybreaks[4]
\newcommand\norm[1]{\left\|#1\right\|}

\def\a{\alpha}
\def\b{\beta}
\def\p{\varphi}
\def\l{\lambda}
\def\g{\gamma}
\def\r{\rho}
\def\k{\kappa}
\def\s{\sigma}

\def\d{\delta}
\def\z{\zeta}
\def\e{\varepsilon}
\def\th{\theta}
\def\P{\Phi}
\def\diag{\mathrm{diag}}
\def\rmd{\mathrm{d}}
\def\NN{\mathbb{N}}

\def\RR{\mathbb{R}}
\def\CC{\mathbb{C}}
\def\DD{\mathbb{D}}
\def\BB{\mathbb{B}}
\def\TT{\mathbb{T}}

\def\FF{\mathbb{F}}
\def\ac{\mathrm{ac}}
\def\rms{\mathrm{s}}
\def\supp{\mathrm{supp\,}}
\def\cH{\mathcal{H}}
\def\klim{K\mbox{-}\lim}

\newtheorem{theorem}{Theorem}[section]
\newtheorem{lemma}[theorem]{Lemma}
\newtheorem{corollary}[theorem]{Corollary}
\newtheorem{proposition}[theorem]{Proposition}
\theoremstyle{definition}
\newtheorem{definition}[theorem]{Definition}
\newtheorem{remark}[theorem]{Remark}

\numberwithin{equation}{section}

\renewcommand{\tilde}{\widetilde}
\renewcommand{\hat}{\widehat}
\renewcommand{\Re}{\mathrm{Re}}
\renewcommand{\Im}{\mathrm{Im}}

\usepackage{xpatch}
\makeatletter   
\xpatchcmd{\@tocline}
{\hfil\hbox to\@pnumwidth{\@tocpagenum{#7}}\par}
{\ifnum#1<0\hfill\else\dotfill\fi\hbox to\@pnumwidth{\@tocpagenum{#7}}\par}
{}{}
\makeatother

\makeatletter
 \def\l@subsection{\@tocline{2}{0pt}{4pc}{6pc}{}}
\def\l@subsubsection{\@tocline{3}{0pt}{8pc}{8pc}{}}
 \makeatother

\title[Orthogonal Polynomials, Verblunsky Coefficients, \&  a S--V Theorem on \(\partial\BB^d\)]{Orthogonal Polynomials, Verblunsky Coefficients, and a Szeg\H{o}--Verblunsky Theorem on the Unit Sphere in \( \CC^d\)}
\date{\today}
\author[C. J. Gauntlett]{Connor J. Gauntlett}
\address{(CJG) School of Mathematics, Statistics and Physics\\
Newcastle University\\
Newcastle upon Tyne NE1 7RU UK}
\email{c.gauntlett@newcastle.ac.uk}

\author[D. P. Kimsey]{David P. Kimsey}
\address{(DPK) School of Mathematics, Statistics and Physics\\
Newcastle University\\
Newcastle upon Tyne NE1 7RU UK}
\email{david.kimsey@newcastle.ac.uk}


\begin{document}

\begin{abstract}
Given a measure $\mu$ on the unit sphere \(\partial\mathbb{B}^d\) in \(\mathbb{C}^d\) with Lebesgue decomposition ${\rm d} \mu = w \, {\rm d} \sigma + {\rm d} \mu_s$, with respect to the rotation-invariant Lebesgue measure \(\s\) on $\partial \mathbb{B}^d$, we introduce notions of orthogonal polynomials $(\varphi_{\alpha})_{\alpha \in \mathbb{N}_0^d}$, Verblunsky coefficients $(\gamma_{\alpha,\beta})_{\alpha,\beta \in \mathbb{N}_0^d}$, and an associated Christoffel function $\lambda_{\infty}^{(d)}(z; {\rm d} \mu)$, and we prove a recurrence relation for the orthogonal polynomials involving the Verblunsky coefficients reminiscent of the classical Szeg\H{o} recurrences, as well as an analogue of Verblunsky's theorem. Moreover, we establish a number of equalities involving the orthogonal polynomials, determinants of moment matrices, and the Christoffel function, and show that if ${\rm supp}\, \mu_s$ is discrete, then the aforementioned quantities depend only on the absolutely continuous part of $\mu$. If, in addition to ${\rm supp}\, \mu_s$ being discrete, one is able to find $f \in H^{\infty}(\mathbb{B}^d)$ such that $f(0) = 1$ and
$$
 \int_{\partial \mathbb{B}^d} |f(\zeta)|^2 w(\zeta) {\rm d}\sigma(\zeta) \leq \exp\left( \int_{\partial \mathbb{B}^d} \log\big(w(\zeta)\big) \, {\rm d}\sigma(\zeta) \right),
 $$
then we establish a \(d\)-variate Szeg\H{o}--Verblunsky theorem, namely
$$
\prod_{\a \in \mathbb{N}_0^d} (1 - \lvert \gamma_{0,\alpha} \rvert^2) = \exp\left(\int_{\partial\mathbb{B}^d} \log\big( w(\zeta) \big) \, {\rm d}\sigma(\zeta)\right).
$$
Finally, we identify several classes of weights where one may construct such an \(f\) and highlight an explicit example of a weight $w$, residing outside of these classes, where $\prod_{\a \in \mathbb{N}_0^d} (1 - \lvert \gamma_{0,\alpha} \rvert^2) \neq \exp\left(\int_{\partial\mathbb{B}^d} \log\big( w(\zeta) \big) \, \rmd\sigma(\zeta)\right)$.
\end{abstract}

\maketitle

\tableofcontents

\section{Introduction}

In the literature, the term ``Szeg\H{o}'s theorem'' is often used to refer to a number of different, but related, results, which intertwine various objects including orthogonal polynomials corresponding to a measure on the unit circle, sequences of points in the unit disc (called Verblunsky coefficients or Schur parameters), a quantity derived from the Christoffel function of a measure, and an entropy quantity. In this paper, we shall focus on the version also called the Szeg\H{o}--Verblunsky theorem: if \(\mu\) is a (non-finitely atomic) probability measure on the unit circle \(\TT\) with Radon--Nikodym derivative (with respect to the normalised Lebesgue measure \(\frac{\rmd\th}{2\pi}\)) \(w\) and Verblunsky coefficients \((\g_n)_{n=0}^{\infty} \subseteq \DD\), then
\[
    \prod_{k=0}^{\infty}(1 - \lvert \g_k \rvert^2) = \exp\left(\int_0^{2\pi} \log(w(e^{i\th})) \, \frac{\rmd\th}{2\pi}\right).
\]
As a consequence, one obtains a characterisation of measures with integrable Radon--Nikodym derivative:
\[
    \sum_{k=0}^{\infty} \lvert \g_k \rvert^2 < \infty \quad \text{if and only if} \quad \int_0^{2\pi} \log\big(w(e^{i\th})\big) \, \frac{\rmd\th}{2\pi} > -\infty,
\]
i.e. such measures are precisely those with square-summable Verblunsky coefficients. Note in particular that while \((\g_n)_{n=0}^{\infty}\) are the Verblunsky coefficients of an arbitrary non-finitely atomic measure \(\mu\) --- which thus may have arbitrary singular part --- the \emph{Szeg\H{o} condition} that \(\frac{1}{2\pi}\int_0^{2\pi} \log\big(w(e^{i\th})\big) \, \rmd\th) > -\infty\) depends only upon the Radon--Nikodym derivative, and hence the absolutely continuous part, of \(\mu\).

Another result sometimes called Szeg\H{o}'s theorem is the weak (or first) Szeg\H{o} limit theorem, which says that if \(\mu\) has Fourier coefficients \((c_n)_{n=-\infty}^{\infty}\) and Toeplitz determinants \(D_n := \det T_n\), where for \(n \in \NN_0\), \(T_n\) is the \((n+1)\times(n+1)\) Toeplitz matrix associated to \(\mu\) by
\[
    T_n := \begin{bmatrix}c_0 & c_1 & \cdots & c_{n-1} \\ c_{-1} & c_0 & \cdots & c_{n-2} \\ \vdots & \vdots & \ddots & \vdots \\ c_{-(n-1)} & c_{-(n-2)} & \cdots & c_0\end{bmatrix} = [c_{k-j}]_{j,k=0}^{n};
\]
then
\begin{equation}
    \label{eq:weakSzego}
    \lim_{n\to\infty} (D_n)^{\frac{1}{n}} = \lim_{n\to\infty} \frac{D_n}{D_{n-1}} = \exp\left(\int_0^{2\pi} \log\big(w(e^{i\th})\big) \, \frac{\rmd\th}{2\pi} \right),
\end{equation}
accordingly, this is related to the Szeg\H{o}--Verblunsky theorem via the \emph{Szeg\H{o} entropy}
\[
    \int_0^{2\pi} \log\big(w(e^{i\th})\big) \, \frac{\rmd\th}{2\pi}.
\]
This may be extended to the \emph{strong} (or second) Szeg\H{o} limit theorem (as in e.g. \cite[Chapter 6]{Sim05a}), a result which considers the second-order asymptotics of the Toeplitz determinants of \(\mu\) when the Verblunsky coefficients of \(\mu\) additionally satisfy
\[
    \sum_{k=0}^{\infty} k\lvert \g_k \rvert^2 < \infty.
\]
In this case, \(\mu\) has singular part \(\mu_{\rms} = 0\), and the strong Szeg\H{o} limit theorem says that
\[
    \log D_n = (n+1) \int_{0}^{2\pi} \log\big(w(e^{i\th})\big) \, \frac{\rmd\th}{2\pi} + \sum_{k=1}^{\infty} k \left\lvert \int_0^{2\pi} e^{-im\th}\log\big(w(e^{i\th})\big) \, \frac{\rmd\th}{2\pi} \right\rvert^2 + o(1).
\]
The strong Szeg\H{o} limit theorem can also be generalised to asymptotics for Toeplitz determinants which have as symbols complex-valued functions, rather than real-valued measures, see, e.g., \cite{Hir66} and \cite{Kac54} and the sharp result \cite{GI71}. 

The Szeg\H{o}--Verblunsky theorem is one of the main results studied in the book \cite{Sim05a}, where four different proofs are provided. Simon goes on to extend the Szeg\H{o}--Verblunsky and weak Szeg\H{o} limit theorems into the following result, \cite[Theorem 2.7.14]{Sim05a}, collecting together many equivalent forms for what we shall call the \emph{Szeg\H{o} quantity} \(\prod_{k=0}^{\infty}(1-\lvert\g_k\rvert^2)\).

Let \(\mu\) be a non-finitely atomic measure on the unit circle \(\TT\) with Lebesgue decomposition with respect to the normalised Lebesgue measure \(\rmd\mu(\th) = w(\th)\frac{\rmd \th}{2\pi} + \rmd \mu_s \). Let \(\mu\) have associated Verblunsky coefficients \((\g_n)_{n=0}^{\infty}\), monic orthogonal polynomials \((\Phi_n)_{n=0}^{\infty}\), orthonormal polynomials \((\p_n)_{n=0}^{\infty}\), Toeplitz determinants \((D_n)_{n=0}^{\infty}\), Szeg\H{o} function \(D(z)\), Christoffel function \(\l_\infty(z; \rmd\mu)\), and Schur function \(f(z)\); let \(\k_n\) be the coefficient of \(z^n\) in \(\p_n\). For a polynomial \(p : \TT \to \CC\), denote its reverse polynomial by \(p^*(z) := z^{\rm{deg}(p)}\overline{p(1/\overline{z})}\). 

\begin{theorem}
    \label{thm:SummaryThm}
    The following quantities are all equal.
    \begin{enumerate}
        \item[\rm{(i)}] \(\exp(\int_{0}^{2\pi}\log(w(\th)) \; \frac{\rmd \th}{2\pi})\);
        \item[\rm{(ii)}] \(\lim_{n\to \infty} \norm{\P_n}^2 = \lim_{n\to\infty} \norm{\P_n^*}^2\);
        \item[\rm{(iii)}] \(\lim_{n\to\infty} \k_n^{-2}\);
        \item[\rm{(iv)}] \(\prod_{k=0}^\infty (1 - \lvert \g_k \rvert^2)\);
        \item[\rm{(v)}] \(\exp(\int_0^{2\pi} \log\left(1 - \lvert f(e^{i \th})\rvert^2\right) \frac{\rmd \th}{2\pi})\);
        \item[\rm{(vi)}] \(\lim_{n\to\infty} \frac{D_{n+1}}{D_n} = \lim_{n\to\infty} \sqrt[n]{D_n}\);
        \item[\rm{(vii)}] \(\l_\infty(0; \rmd\mu)\);
        \item[\rm{(viii)}] \(1 - \norm{Q_+ z^{-1}}^2, \text{ where \(Q_+\) is the projection onto \(\mathrm{span}\{z^j\}_{j=0}^\infty\)}\);
        \item[\rm{(ix)}] \(\lvert D(0) \rvert^2\);
        \item[\rm{(x)}] \(\lim_{n\to\infty} \lvert \p_n^*(0) \rvert^{-2}\);
        \item[\rm{(xi)}] \(\lim_{n\to\infty} \left(\sum_{j=0}^n \lvert \p_j(0) \rvert^2\right)^{-1}\);
        \item[\rm{(xii)}] The relative entropy \(S(\frac{\rmd \th}{2\pi} \; \vert \; \rmd \mu)\).
    \end{enumerate}
\end{theorem}

Notice in particular the equalities (i) = (iv) and (i) = (vi), which are the Szeg\H{o}--Verblunsky and weak Szeg\H{o} limit theorems respectively. We defer to \cite{Sim05a} for an extensive treatment of orthogonal polynomials on the unit circle and a comparison of this theory with that of orthogonal polynomials on the real line, and to \cite{Sim05b} for an expansion upon this comparison.

Classically, the Verblunsky coefficients of a measure \(\mu\) arise in two ways, unified by Geronimus' theorem: firstly, if \(F : \DD \to \CC\) is the Herglotz transform of \(\mu\) given by
\[
    F(z) = \int_0^{2\pi} \frac{1 + ze^{-i\th}}{1 - ze^{-i\th}} \, \rmd\mu(e^{i\th})
\]
and \(f : \DD \to \overline{\DD}\) is the Schur function given by 
\[
    F(z) = \frac{1 + f(z)}{1 - f(z)},
\]
then the \emph{Schur parameters} \((\g_n)_{n=0}^{\infty}\) of \(\mu\) are the outputs of the Schur algorithm beginning with \(f_0 = f\). On the other hand, if \((\p_n)_{n=0}^{\infty}\) are the orthonormal polynomials associated to \(\mu\) via Gram--Schmidt on \(\{1, z, z^2, \ldots\}\) with respect to the inner product
\[
    \langle p, q \rangle_\mu := \int_0^{2\pi} p(e^{i\th}) \overline{q(e^{i\th})} \, \rmd\mu(e^{i\th}),
\]
then the \emph{Verblunsky coefficients} of \(\mu\) can be viewed as the coefficients \((\a_n)_{n=0}^{\infty}\) appearing in the Szeg\H{o} recurrences
\[
    \p_{n+1}(z) = \frac{z\p_n(z) - \overline{\a_n} \p_n^{\#}(z)}{\sqrt{1 - \lvert \a_n \rvert^2}}
\]
and
\[
    \p_{n+1}^{\#}(z) = \frac{\p_n^{\#}(z) - \a_n z \p_n(z)}{\sqrt{1 - \lvert \a_n \rvert^2}},
\]
where for a polynomial \(p\) we denote by \(p^{\#}\) the \emph{reverse polynomial} given by
\[
    p^{\#}(z) := z^{\mathrm{deg}(p)}\overline{p(1/\overline{z})}.
\]
Geronimus' theorem \cite[Theorem 3.1.4]{Sim05a} states that \(\a_n = \g_n\) for all \(n\), and so we use the terminology ``Verblunsky coefficients'' to refer to either in accordance with the vocabulary of \cite{Sim05a}.

As Simon discusses in \cite{Sim05a} and \cite{Sim05b}, much of this theory has an analogue for measures on (a compact interval of) the real line. One useful tool for establishing a link between the the unit circle and real line settings is the {\it Szeg\H{o} mapping} (see Section 13.1 in \cite{Sim05b}), which establishes a bijective correspondence between probability measures on the unit circle and probability measures on the interval $[-2,2]$. In this setting, when \(\nu\) is a probability measure on \([-2,2]\), the orthonormal polynomials \((p_n)_{n=0}^{\infty}\) of \(\nu\) obey three term recurrences involving particular coefficients \((a_n)_{n=0}^{\infty}\), \((b_n)_{n=0}^{\infty}\) (which appear in the \emph{Jacobi matrix}):
\begin{equation}
    \label{eq:JacobiRecurrence}
    xp_n(x) = a_{n+1} p_{n+1}(x) + b_{n+1} p_n(x) + a_n p_{n-1}(x).
\end{equation}
In the real setting, Toeplitz matrices are replaced by Hankel matrices, i.e. matrices of the form $[s_{j+k}]_{j,k=0}^n$, and if \(\nu\) has Lebesgue decomposition $\rmd\nu = f(x) \rmd x + \rmd\nu_{\rms}$ then the Szeg\H{o}--Verblunsky theorem takes the form
\begin{equation}
    \label{eq:OPRLSzego}
    \lim_{n \to \infty} \prod_{j=1}^n a_j = \sqrt{2} \exp\left(\int_0^{2\pi} \log \big( 2\pi \sin(\theta) f(2 \cos(\theta)) \big) \, \frac{\rmd\th}{2\pi} \right).
\end{equation}

Moving to several variables, if one replaces the unit disc by the unit ball \(\BB^d\) or polydisc \(\DD^d\) in \(\CC^d\), or (a compact interval in) the real line by (a compact subset of) $\RR^d$, then generalising these ideas becomes delicate. For a detailed and authoritative treatment of the theory of orthogonal polynomials on $\RR^d$, see \cite{DX01}. In particular, it is worth pointing out that in the several variable setting \eqref{eq:JacobiRecurrence} has a generalisation where the Jacobi coefficients are replaced by matrix-valued coefficients, see Section 3.2 in \cite{DX01} for details. 

In the ground-breaking work \cite{Oki96}, Okikiolu discovered a weak and strong Szeg\H{o} limit theorem for continuous functions on the real unit sphere in $\RR^3$ and $\RR^4$ with the additional hypotheses that the function belongs to a fractional Sobolev space and the closed convex hull of the image of the function does not contain the origin. 

Meanwhile in the complex multivariate setting, study of the Szeg\H{o}--Verblunsky theorem and weak Szeg\H{o} limit theorem has been done primarily on the polytorus, i.e. replacing the unit disc \(\DD\) and circle \(\TT\) by, respectively, the polydisc $\DD^d := \{ (z_1, \ldots, z_d): \text{$|z_j| < 1$ for $j=1,\ldots, d$}$ and polytorus $\TT^d := \{ (\zeta_1, \ldots, \zeta_d): \text{$|\zeta_j| = 1$ for $j=1,\ldots, d$} \}$: this goes back at least to the seminal works of Helson and Lowdenslager \cite{HL58} and \cite{HL61}. An abstract Szeg\H{o} theorem in terms of an analogue of a Christoffel function --- specifically, a result analogous to the equality (i) = (vii) of \Cref{thm:SummaryThm} in far more generality --- appeared in Browder \cite{Bro69}. To elaborate: let $X$ be a compact Hausdorff space and $C(X)$ denote the algebra of continuous functions \(X \to \CC\) equipped with the supremum norm $\| \cdot \|_{\infty}$. We say that $\mathfrak{A} \subseteq C(X)$ is a {\it sup-norm algebra} if $\mathfrak{A}$ is a $\| \cdot \|_{\infty}$-closed complex subalgebra, $1 \in \mathfrak{A}$ and $\mathfrak{A}$ separates points of $X$. Let $\hat{\mathfrak{A}}$ denote the set of multiplicative linear functions on $\mathfrak{A}$ in the weak-* topology. We say that $\varphi \in \hat{\mathfrak{A}}$ has a {\it unique representing measure}, say $\mu_{\varphi}$, if $\mu_{\varphi}$ is a probability measure on $X$ with the property that
\begin{equation}
    \label{eq:unique}
    \varphi(f) = \int_X f(x) \, {\rm d}\mu_{\varphi}(x) \quad {\rm for} \quad f \in \mathfrak{A}.
\end{equation}
Browder \cite{Bro69} showed that for any sup-norm algebra $\mathfrak{A} \subseteq C(X)$ and \(\p \in \hat{\mathfrak{A}}\) with a unique representing measure, an abstract Szeg\H{o} theorem holds which we will now formulate. 
\begin{theorem}
\label{thm:Browder}
    Suppose $X$ is a compact Hausdorff space. Let \(\mathfrak{A} \subseteq C(X)\) be a a sup-norm algebra. Suppose \(\p \in \hat{\mathfrak{A}}\) has a unique representing measure \(\rmd\mu_\p\) on \(X\). Then for any measure \(\rmd\nu\) on \(X\) with Lebesgue decomposition \(\rmd\nu = w \rmd\mu_{\p} + \rmd\nu_\rms\) with \(\nu_\rms\) singular with respect to \(\mu_\p\), we have for any \(p \in (0,\infty)\) that
    \[
        \inf\left(\int_X \lvert P(x) \rvert^p \, \rmd \nu(x) : P \in \mathfrak{A}, \; {\rm and} \; \p(P) = 1 \right) = \exp\left(\int_X \log w(x) \, \rmd\mu_\p(x)\right).
    \]
\end{theorem}
The expression on the left hand side is directly analogous to an \(L^p\)-generalisation of the Christoffel function associated to a measure on \(\TT\), and hence as mentioned above this result is a very general analogy of \Cref{thm:SummaryThm}, (i) = (vii). We refer to the discussion surrounding \cite[Theorem 2.5.5]{Sim05a} for further details.

More recently, work in the polytorus setting has admitted Szeg\H{o} theorems in different guises. The paper \cite{Gib23} introduces a \(\NN\)-indexed, \(\CC\)-valued notion of Verblunsky coefficients that may be associated to a measure on the polytorus $\TT^d$ and obtains a corresponding Szeg\H{o}--Verblunsky theorem, and in \cite{GLZ24} a weak Szeg\H{o} limit theorem was obtained on the polytorus taking the perspective of studying determinants of Toeplitz operators with positive symbols in \(L^1(\TT^d)\) bounded away from 0.

On the unit sphere \(\partial\BB^d\) in \(\CC^d\), less is known, though we remark that one may re-interpret the \(\RR^4\) setting of \cite{Oki96} as giving Szeg\H{o} limit theorems on the unit ball in \(\CC^2\).

The notion of a Schur algorithm on the unit ball has been studied depending on the application, for example in \cite{ABK02} from the perspective of interpolation problems; note that in this context the outputs of the algorithm are vector-valued, lying in the unit ball rather than the unit disc. In the book \cite{Con96}, Constantinescu presents an alternative to the Schur algorithm using deep and general matrix theory: if \((\cH_n)_{n=0}^{\infty}\) is a sequence of Hilbert spaces and \(K\) is a kernel with \(K(n,m) \in B(\cH_n,\cH_m)\) for \(n,m \in \NN_0\), then \cite[Theorem 1.5.3]{Con96} constructs a family of Verblunsky coefficients associated to \(K\) via Douglas' lemma. Note that by Theorem 1.6.7 of the same source and Geronimus' theorem, these Verblunsky coefficients recover the usual notion when \(\cH_n = \CC\) for all \(n\).

The setting of \cite{Con96} is incredibly general, and as such these Verblunsky coefficients may be applied widely. Success has already been found in this paradigm in the free noncommutative setting: in 2002, Constantinescu and Johnson \cite{CJ02b} associated Verblunsky coefficients and orthogonal polynomials to a \emph{multi-Toeplitz kernel} on the free monoid on \(d\) generators, and the present authors built on this in \cite{GK25} by linking such kernels to the noncommutative measures of Jury and Martin (see e.g. \cite{JM22a} and references therein) to establish an analogue of many items of \Cref{thm:SummaryThm} with \(\mu\) replaced by a noncommutative measure. In particular, the dependence of the square-summability of the Verblunsky coefficients upon only the absolutely continuous part of \(\mu\) is retained in this setting. In the sequel, we shall apply similar methods to the setting of measures on the unit sphere \(\partial\BB^d\) in \(\CC^d\) and again obtain an analogue of \Cref{thm:SummaryThm}. Moreover --- further than what was obtained in the free noncommutative setting --- we shall prove a full Szeg\H{o}--Verblunsky theorem under particular assumptions on the measure, construct classes of examples satisfying these assumptions, and highlight a counterexample to the Szeg\H{o}--Verblunsky theorem in several variables when these assumptions are not met.

\subsection{Main Conclusions and Paper Structure}

Below are our main conclusions.

\begin{enumerate}
    \item[(C1)] Given a probability measure \(\mu\) on \(\partial\BB^d\), we construct a multi-sequence of orthogonal polynomials on \(\partial\BB^d\), denoted by $(\varphi_{\alpha})_{\alpha \in \NN_0^d}$, and an analogue of Verblunsky coefficients (i.e., a multisequence of points in $\overline{\DD}$) denoted by $(\gamma_{\alpha, \beta})_{\alpha, \beta \in \NN_0^d}$ which are intertwined by a Szeg\H{o}-type recurrence relation (see \Cref{def:VCs}, \Cref{def:OPs} and \Cref{thm:Recurrences}, respectively).

    \begin{remark}
        The reader should note that the polynomials $(\varphi_{\alpha}^{\#} )_{\alpha \in \NN_0^d}$, i.e., the {\it sharp polynomials}, appearing in \Cref{thm:Recurrences} need not agree with the usual understanding when $d=1$, see \Cref{rem:SharpPolys} for more details.
    \end{remark}

    \item[(C2)] We introduce a Christoffel function associated to a probability measure on \(\partial\BB^d\) and prove an infimum formula for this function (see \Cref{lem:CFExists}).

    \item[(C3)] We establish a multivariate analogue of \Cref{thm:SummaryThm} for a probability measure on \(\partial\BB^d\), though in the multivariate setting, this list becomes two, with the first providing an analogue of the weak Szeg\H{o} limit theorem and the second providing an analogue of the Szeg\H{o}--Verblunsky theorem (see \Cref{thm:MainResult}, (iii) = (iv) and (v) = (viii)).

    \begin{remark}
    \label{rem:NOTAPPLICABLE}
    We note for the convenience of the reader that no $z \in \BB^d$, when $d > 1$, has a unique representing measure (see \cite[p. 186]{Rud87}) and hence Theorem \ref{thm:Browder} is not applicable to the main setting of this paper.
    \end{remark}

    \item[(C4)] When the singular part of a probability measure \(\mu\) on \(\BB^d\) is discrete, the Christoffel function associated to \(\mu\) --- and hence the Szeg\H{o} quantity \(\prod_{\a \in \NN_0^d} (1 - \lvert \a \rvert^2)\) --- depends only upon the absolutely continuous part of \(\mu\) (see \Cref{thm:AbsolutelyContinuousDeterminesEntropy} and \Cref{thm:MainResult}, (v) = (viii)').

    \item[(C5)] In addition to the assumption of (C4), let \(\mu\) have Lebesgue decomposition \(\rmd\mu(\z) = w(\z)\rmd\s(\z) + \rmd\mu_{\rms}(\z)\). We show that when there exists \(f \in H^{\infty}(\BB^d)\) with \(f(0) = 1\) and such that
    \[
        \int_{\partial\BB^d} |f(\z)|^2 w(\z) \, \rmd\s(\z) \leq \exp\left(\int_{\partial\BB^d} \log\big( w(\z) \big) \, \rmd\s(\z)\right)
    \]
    then we have a full analogue of the Szeg\H{o}--Verblunsky theorem, i.e. an analogue of the equality (i) = (v) of \Cref{thm:SummaryThm} (see \Cref{thm:MainResult}, (v) = (x)).

    \item[(C6)] We study a number of classes of functions giving rise to some \(w \in L^1(\partial\BB^d)\) where the assumption of (C5) is satisfied (e.g., $w = |g|^2$, where $g, 1/g \in H^{\infty}(\BB^d)$ and $\| g \|_{\infty} \leq 1$) and hence the Szeg\H{o}--Verblunsky theorem holds (see \Cref{thm:examplesinfty}, \Cref{thm:examplesp}, \Cref{thm:examplesinftyapprox} and \Cref{rem:FurtherExamples}). We also provide a example when the equality (x) = (v) does not hold when $w = |g|^2$ and $g$ does not belong to one of the aforementioned classes in (C5) (see \Cref{thm:counterexample}).

\end{enumerate}

\medskip

The paper is structured as follows.

\medskip

The remainder of the introduction is dedicated to fixing our notation for  the indexing set \(\NN_0^d\) and some common spaces of functions on \(\BB^d\) and \(\partial\BB^d\). In Section 2 we discuss the relevant parts of the kernel theory of \cite{Con96} and apply this theory to introduce Verblunsky coefficients associated to a probability measure on \(\partial\BB^d\). These ideas are intrinsically linked to those of Section 3, where we establish (C1). In Section 4, we discuss the Christoffel function of a measure and show (C2). Sections 5 and 6 are devoted to proving (C3), (C4) and (C5). Section 7 returns to the assumptions on the absolutely continuous part of the measure from Section 5 and discusses (C6) in detail.

\subsection{Multi-indices}

In our multivariate setting we shall often index families of objects by elements of \(\NN_0^d\); we call such elements \emph{multi-indices}. The \emph{length} of a multi-index \(\a = (\a_1, \ldots, \a_d) \in \NN_0^d\) is given by
\[
    \lvert \a \rvert = \a_1 + \cdots + \a_d;
\]
we shall also use the notation
\[
    \a! = \a_1! \cdot \cdots \cdot \a_d!.
\]
The set \(\NN_0^d\) is a commutative monoid under vector addition with identity \(0 = (0, \ldots, 0)\) and generators \(e_j\), where $e_j$ denotes the $d$-tuple of zeros with a 1 in the $j^{\mathrm{th}}$ position for \(j = 1, \ldots, d\).

We order \(\NN_0^d\) via the \emph{short lexicographical} (\emph{shortlex}) ordering, that is, the ordering on \(\NN_0^d\) such that \(\a \prec \b\) if \(\lvert \a \rvert < \lvert \b \rvert\) and when \(\lvert \a \rvert = \lvert \b \rvert\) we order lexicographically such that \(e_1 \prec e_2 \prec \ldots \prec e_d\). For example, when \(d = 2\), this ordering is
\[
    \NN_0^2 = \{(0,0), (1,0), (0,1), (2,0), (1,1), (0,2), \ldots\}.
\]

We shall use \(\mathrm{prec}(\a)\) and \(\mathrm{succ}(\a)\) to refer to the immediate predecessor and successor, respectively, of \(\a\) according to this ordering, and for \(n \in \NN_0\) we denote by \(\a(n)\) be the last element of \(\NN_0^d\) of length \(n\) under the shortlex ordering, i.e. \(\a(n) = (0, \ldots, 0, n) = n e_d\).

Let \(\a = (\a_1, \ldots, \a_d) \in \NN_0^d\) and \(z = (z_1, \ldots, z_d) \in \CC^d\). We define the monomial \(z^\a\) to be
\[
    z^\a := z_1^{\a_1}\cdot \cdots \cdot z_d^{\a_d}.
\]
The short lexicographical ordering on \(\NN_0^d\) induces an ordering on the set of monomials \(\{z^{\a} : \a \in \NN_0^d\}\), given by \(z^{\a} \prec z^{\b}\) if and only if \(\a \prec \b\); we also call this the \emph{short lexicographical ordering}.

We define the \emph{multi-degree} of a \(d\)-variate polynomial \(p\), written as
\[
    p(z) = \sum_{\a \in \NN_0^d} c_\a z^{\a}
\]
with \((c_{\a})_{\a\in\NN_0^d}\) eventually zero, to be the shortlex-highest \(\a \in \NN_0^d\) such that \(c_{\a} \neq 0\).

\subsection{Function Spaces}

We often refer to a number of standard sets and function spaces; to aid the reader, we state our notation here.

\begin{itemize}
    \item The open unit ball in \(\CC^d\) is \(\BB^d\) with boundary the unit sphere \(\partial\BB^d\). When \(d = 1\) we write \(\DD\) in place of \(\BB^1\) and \(\TT\) in place of \(\partial\BB^1\).
    \item \(\RR[x_1, \ldots, x_d]\) consists of polynomials in \(d\) (real) variables with real coefficients.
    \item \(\CC[z_1, \ldots, z_d]\) is the space of polynomials in \(d\) variables with complex coefficients.
    \item For \(\a \in \NN_0^d\), \(\CC[z_1, \ldots, z_d]_{\a}\) is the space \(\mathrm{span}\{z^{\b}\}_{0\preceq \b \preceq \a}\) of polynomials in \(d\) variables of multi-degree at most \(\a\).
    \item The space \(C(\partial\BB^d)\) is the space of all continuous functions \(\BB^d \to \CC\).
    \item Given a measure \(\mu\) on \(\partial\BB^d\) and \(1 \leq p < \infty\), the space \(L^p(\mu)\) is the space of measurable functions \(f : \partial\BB^d \to \CC\) such that
    \[
        \int_{\partial\BB^d} \lvert f(\z) \rvert^p \, \rmd\mu(\z) < \infty.
    \]
    When \(\mu = \s\) is the normalised rotation-invariant Lebesgue measure on \(\partial\BB^d\), we write \(L^p(\mu) = L^p(\partial\BB^d)\).
    \item Given \(1 \leq p < \infty\), we let \(H^p(\BB^d)\) be the space of holomorphic \(f : \BB^d \to \CC\) such that
    \[
        \sup_{0 < r < 1} \int_{\partial\BB^d} \lvert f(r\z) \rvert^p \, \rmd \s(\z) < \infty
    \]
    where \(\s\) is the normalised rotation-invariant Lebesgue measure on \(\partial\BB^d\).
    \item When \(p = \infty\) we let \(H^{\infty}(\BB^d)\) be the space of bounded holomorphic functions \(\BB^d\to\CC\).
\end{itemize}

\section{Kernels and Verblunsky Coefficients}

When \(d = 1\), we have the usual relationship between positive Toeplitz kernels and positive measures, and furthermore, by definition,
\[
    \mu(\TT) = \int_{\TT} 1 \, \rmd \mu = \int_0^{2\pi} e^{ik\th} e^{-ik\th} \, \rmd\mu(e^{i\th}) = K(k,k)
\]
for any \(k\in\NN_0\). It follows that not only is the kernel associated to a measure on \(\TT\) constant along the diagonal, one also has that probability measures correspond precisely to \emph{normalised} positive kernels, in the sense that \(K(n,n) = 1\) for all \(n\in\NN_0\).

For \(d > 1\), a similar fact for noncommutative measures (in the sense of \cite{JM22a}) was shown in \cite{GK25}, namely that nc measures correspond one-to-one with \emph{multi-Toeplitz} kernels. Moreover, the analogy of \(\mu(\TT)\) in that setting is \(\mu(1)\) for an nc measure \(\mu\) (and the \emph{constant nc function} 1 given by \(1(Z) = I_{k\times k}\) for \(Z \in (\CC^{k\times k})^d\)), and we see that
\[
    \mu(1) = \mu\big((L^{\s})^{\ast}L^{\s}\big) = K(\s,\s)
\]
for any \(\s \in \FF_d^+\). Hence multi-Toeplitz kernels are similarly constant along the ``diagonal". This may also be seen from the axioms defining a multi-Toeplitz kernel; we take the above perspective to invite the comparison to the univariate theory. Thus we have a bijective relationship also between probability nc measures (\(\mu(1) = 1\)) and normalised multi-Toeplitz kernels.

In this paper we utilise the analogues of these ideas for the commutative multivariate setting, where we again would like to study kernels corresponding to positive measures, now on the boundary \(\partial\BB^d\) of the unit ball of \(\CC^d\). The desired relationship would be a one-to-one correspondence between such measures \(\mu\) and some class of kernels \(K\) on \(\NN_0^d\) via the moments of the measure, i.e. via the formula
\[
    K(\a, \b) = \int_{\partial\BB^d} \z^{\a} \overline{\z}^{\b} \, \rmd\mu(\z).
\]

\begin{remark}
    \label{rem:multiToeplitz}
    It is natural, beginning from the perspective of adapting the noncommutative theory of \cite{CJ02b} and \cite{GK25}, to first consider positive kernels satisfying a commutative version of the multi-Toeplitz axioms:
    \[
        K(\a + \a', \b + \a') = K(\a,\b), \quad \a,\b,\a' \in \NN_0^d,
    \]
    \[
        K(\a, \b) = 0 \quad \text{ if there is no \(\a' \in \NN_0^d\) such that \(\a = \b + \a'\) or \(\b = \a + \a'\).}
    \]
    One quickly observes that this is highly restrictive. To illustrate this, consider the case \(d = 2\) and let \(\a = (1,0), \b = (0,1)\): since there is no \(\a' \in \NN_0^d\) such that \((1,0) = (0,1) + \a'\) or vice versa, the second of the above equations forces \(K(\a,\b) = 0\) for any kernel in this class, so that
    \[
        \int_{\partial\BB^d} \z_1 \overline{\z_2} \, \rmd\mu(\z_1,\z_2) = 0
    \]
    for any measure corresponding to a kernel satisfying these axioms. It follows that this restriction of the multi-Toeplitz axioms in \cite{GK25} to commutative variables does not provide the class of kernels in which we shall be interested in this paper.
\end{remark}

As the relationship between kernels on \(\NN_0^d\) and measures on \(\partial\BB^d\) shall be important to us, we first elucidate the connection. The perspective we shall take from now on is that we begin with a measure on \(\partial\BB^d\), from which we shall derive a kernel and thence orthogonal polynomials and Verblunsky coefficients. Since $\partial \BB^d$ can be viewed a compact semi-algebraic set, the converse direction, i.e. the question of when a given kernel on \(\NN_0^d\) defines a genuine measure on \(\partial\BB^d\), is subsumed by the more general work of Putinar in the setting of quadratic modules \cite{Put93}. For the convenience of the reader, we will make this connection explicit. 

\begin{theorem}
\label{thm:Kernel-Measure}
    Let \(K\) be a positive kernel on \(\NN_0^d\). There exists a positive measure \(\mu\) on \(\partial\BB^d\) with moment kernel \(K\), i.e. that
    \[
        K(\a,\b) = \int_{\partial\BB^d} \z^{\a} \overline{\z}^{\b} \, \rmd\mu(\z)
    \]
    if and only if
    \[
        K(\a, \b) = \sum_{j=1}^{d} K(\a + e_j, \b + e_j)
    \]
    for all \(\a, \b \in \NN_0^d\).
\end{theorem}
\begin{proof}
    We appeal to \cite[Theorem 3.3]{Put93}; note that in our setting, the \(n\) therein is equal to \(2d\). The unit sphere is the compact (algebraic, hence) semi-algebraic set characterised by the polynomials (in \(z\) and \(\overline{z}\))
    \[
        p(z,\overline{z}) = \sum_{j=1}^d z_j \overline{z}_j - 1, \quad q(z, \overline{z}) = 1 - \sum_{j=1}^d z_j \overline{z}_j,
    \]
    so there exists a positive measure \(\mu\) on \(\partial\BB^d\) with moment matrix \(K\) if and only if the linear functional $L: \CC[z_1, \ldots, z_d, \bar{z}_1, \ldots, \bar{z}_d] \to \CC$ given by by $L(z^{\alpha} \bar{z}^{\beta} ) = K(\alpha, \beta)$ for $\a,\b \in \NN_0^d$ (and extended by linearity) satisfies
    \begin{equation}
        \label{eq:positivityConditions}
    L(|u|^2) \geq 0, \; L(|u|^2 p) \geq 0, \; {\rm and} \; L(|u|^2 q) \geq 0 \quad \text{ for } \quad u \in \CC[z_1, \ldots, z_d, \bar{z}_1, \ldots, \bar{z}_d].
    \end{equation}
    The first condition, that $L(|u|^2) \geq 0$ for all $u \in \CC[z_1, \ldots, z_d, \bar{z}_1, \ldots, \bar{z}_d]$, is equivalent to positivity of \(K\) and hence is true by assumption. Note that \(p = -q\); hence the second and third positivity conditions together imply that \(L(|u|^2 p) = 0\) and hence $L(|u|^2 q) = 0$, i.e. \Cref{eq:positivityConditions} all hold if and only if \(K\) is positive and \(L^2(\lvert u \rvert^2 p) = 0\) for \(u \in \CC[z_1, \ldots, z_d, \bar{z}_1, \ldots, \bar{z}_d]\).
    
    Compute directly that
    \[
        L(z^{\alpha}\bar{z}^{\beta} p(z,\bar{z})) = \sum_{j=1}^d K(\a+e_j, \b+e_j) - K(\a,\b),
    \]
    so that $L(|u|^2 p) = 0$ for all $u \in \CC[z_1, \ldots, z_d, \bar{z}_1,\ldots, \bar{z}_d]$ if and only if
    \[
        K(\a,\b) = \sum_{j=1}^d K(\a + e_j, \b + e_j)
    \]
    for all \(\a, \b \in \NN_0^d\).
\end{proof}

\begin{remark}
    \label{rem:Normalisation}
    The conditions on \(K\) arising from \Cref{thm:Kernel-Measure} also shows that normalisation becomes more complex. Given a positive kernel \(K\), the kernel
    \[
        \tilde{K}(\a,\b) := \frac{K(\a,\b)}{\sqrt{K(\a,\a)}\sqrt{K(\b,\b)}},
    \]
    satisfies \(\tilde{K}(\a,\a) = 1\) for all \(\a \in \NN_0^d\), so is normalised in the same sense as we saw in the univariate and noncommutative settings. However, one consequence of \Cref{thm:Kernel-Measure} is that, for \(\tilde{K}\) to give rise to a measure on \(\partial\BB^d\), taking \(\a = \b = 0\) we must have that
    \[
        \tilde{K}(0,0) = \sum_{j = 1}^d \tilde{K}(e_j, e_j);
    \]
    since \(\tilde{K}\) is definitionally normalised, this becomes \(d = 1\). Hence \(\tilde{K}\) can only be a kernel of moments of a measure on \(\partial\BB^d\) if \(d = 1\), and when \(d > 1\) one cannot simply normalise a kernel and retain the link to measure theory. 
    
    As we shall see, this obstruction shall make later computations more technical, but not completely inaccessible.
\end{remark}

We may now apply the kernel theory of \cite{Con96} to associate a family \((\g_{\a,\b})_{\a,\b\in\NN_0^d}\) of complex numbers inside \(\overline{\DD}\), called \emph{Verblunsky coefficients}, to a measure on \(\partial\BB^d\).

\begin{definition}[Construction of Verblunsky Coefficients] 
    \label{def:VCs}
    Let \(\mu\) be a positive measure on \(\partial\BB^d\).
    \begin{enumerate}[align=left]
        \item[\textbf{Step 1:}] Define a kernel \(K\) on \(\NN_0^d\) by
        \[
            K(\a, \b) = \int_{\partial\BB^d} \z^{\a} \overline{\z}^{\b} \, \rmd\mu(\z)
        \]
        for \(\a,\b \in \NN_0^d\). 
    
        \item[\textbf{Step 2:}] Since \(\mu\) is positive, the matrix
        \[
            [K(\a,\b)]_{\a,\b \preceq \a_0} \in \CC^{(\lvert \a_0 \rvert + 1) \times (\lvert \a_0 \rvert + 1)}
        \]
        is positive semi-definite for all \(\a_0 \in \NN_0^d\), and the shortlex ordering thus gives us a positive kernel \(\tilde{K}\) on \(\NN_0\) via
        \[
            \tilde{K}(n,m) = K(\a,\b)
        \]
        where \(\a,\b\) are respectively elements \(n\) and \(m\) of \(\NN_0^d\) according to the shortlex ordering.
        
        \item[\textbf{Step 3:}] By \cite[Theorem 1.5.3]{Con96} there exists a unique family \((\g_{n,m})_{n,m\in\NN_0}\) associated to \(\tilde{K}\), with \(\g_{n,n+1} = K(n, n+1)\) and \(\lvert \g_{n,m} \rvert \leq 1\) for all \(n,m\in\NN\).
        
        \item[\textbf{Step 4:}] Once more applying the shortlex ordering, we obtain a family \((\g_{\a,\b})_{\a,\b \in \NN_0^d}\) where \(\g_{\a,\b} := \g_{n,m}\) for \(\a\) the \(n^{\text{th}}\) element of \(\NN_0^d\) in the shortlex ordering and \(\b\) the \(m^{\text{th}}\).
    \end{enumerate}
    We call \((\g_{\a,\b})_{\a,\b\in\NN_0^d}\) the \emph{Verblunsky coefficients} of \(\mu\).
\end{definition}

\begin{remark}
    Strictly speaking, we may only apply \cite[Theorem 1.5.3]{Con96} to \(\tilde{K}\) when \(\tilde{K}\) is normalised, which will in general not be the case as discussed in \Cref{rem:Normalisation}. However, the result in fact holds without this assumption, as elaborated upon in \cite[Remark 1.5.4]{Con96}, where a formula for this more general correspondence is provided.
\end{remark}

Accordingly, we define a family \((d_{\a,\b})_{\a,\b \in \NN_0^d}\) by \(d_{n,m} := \sqrt{1 - \lvert \g_{n,m} \rvert^2}\) and \(d_{\a,\b} = d_{n,m}\) when \(\a\) is the \(n^{\text{th}}\) element of \(\NN_0^d\) in the short lexicographical ordering and \(\b\) is the \(m^{\text{th}}\).

\begin{remark}
    At this stage in the noncommutative study \cite{CJ02b}, Constantinescu and Johnson apply the multi-Toeplitz structure of the kernel to see that the Verblunsky coefficients are uniquely determined by a one-parameter sub-family. As we discussed in \Cref{rem:multiToeplitz}, we no longer have this structure, and accordingly in this setting we do not have an analogous result. We shall see later that we are nevertheless able to obtain analogies of many results obtained in the noncommutative setting, modulo some careful handling of the bi-indexed family and the non-normalised kernel.
\end{remark}

We now introduce a technical assumption on our measures that shall be assumed throughout the remainder of the paper.

\begin{definition}
\label{def:nontrivial}
    Let \(\mu\) be a positive measure on \(\partial\BB^d\) with moment kernel \(K\). We say that \(\mu\) is \emph{non-trivial} if the matrix \([K(\a',\b')]_{0 \preceq \a', \b' \preceq \a}\) is invertible for all \(\a \in \NN_0^d\).
\end{definition}

\begin{remark}
    When \(d = 1\), non-triviality of a measure is equivalent to the measure not being finitely atomic. However, for \(d > 1\), this is no longer the case: a trivial measure may be supported on an infinite set in $\partial \BB^d$.
\end{remark}

The non-triviality condition in Definition \ref{def:nontrivial} also arose in the free noncommutative setting \cite{GK25}; a motivation for this assumption from the univariate setting was given in Remark 4.2 therein.

\section{Orthogonal Polynomials on the Unit Sphere in \(\CC^d\)}

In this section we associate to a measure on the unit sphere in \(\CC^d\) a bi-indexed family of polynomials orthogonal with respect to the inner product induced by the measure, though we shall primarily be interested in a uni-indexed subset of these. Motivated by the work of Constantinescu \cite{Con96} and of Constantinescu and Johnson \cite{CJ02b}, we obtain a recurrence for these polynomials involving their Verblunsky coefficients, analogous to the Szeg\H{o} recurrences for orthogonal polynomials on the unit circle.

A measure \(\mu\) on \(\partial\BB^d\) induces a pre-inner product on polynomials in the usual sense:
\[
    \langle p,q \rangle_{\mu} := \int_{\partial\BB^d} p(\z) \overline{q(\z)} \,\rmd\mu(\z).
\]
Note that when \(\mu\) is non-trivial, this will be a genuine inner product, i.e. positive definite.

\begin{definition}
    \label{def:OPs}
    Let \(\mu\) be a non-trivial probability measure on \(\partial \BB^d\). We perform the Gram--Schmidt algorithm in \(\CC[z_1, \cdots, z_d]\) on monomials ordered according to the short lexicographical ordering to obtain \emph{orthonormal polynomials} \((\p_{\a})_{\a \in \NN_0^d}\), and Gram--Schmidt without normalisation to obtain \emph{orthogonal monic polynomials} \((\P_{\a})_{\a \in \NN_0^d}\). We write \(a_{\a,\b}\) for the \(\b^{\text{th}}\) coefficient of \(\p_{\a}\) so that
    \[
        \p_{\a}(z) = \sum_{\b = 0}^{\a} a_{\a,\b} z^{\b}.
    \]
\end{definition}

\begin{remark}
    Note that Gram--Schmidt definitionally provides a polynomial with leading term \(z^{\a}\) for each \(\a \in \NN_0^d\), so that the orthogonal and orthonormal polynomials associated to \(\mu\) are \(\p_\a\) and \(\P_\a\) of degree \(\a\).
\end{remark}

We shall next outline the construction connecting these polynomials to the kernel theory of Constantinescu \cite{Con96}.

\medskip

Given polynomials \(p(z) = \sum_{\a} p_{\a}z^{\a}\), \(q(z) =  \sum_{\a} q_{\a}z^{\a}\), let \(\hat{p} = \begin{bmatrix}p_{\a}\end{bmatrix}_{\a}\), \(\hat{q} = \begin{bmatrix}q_{\a}\end{bmatrix}_{\a}\) be their coefficient vectors ordered according to the shortlex ordering. Notice
\begin{align*}
    \langle p, q\rangle_\mu & = \int_{\partial\BB^d} p(z) \overline{q(z)} \, \rmd\mu(z) \\
    & = \int_{\partial\BB^d} \sum_{\a,\b} p_{\b} \overline{q_{\a}} z^{\b} \overline{z^{\a}} \, \rmd\mu(z) \\
    & = \sum_{\a,\b} p_{\b} \overline{q_{\a}} \int_{\partial\BB^d} z^{\b} \overline{z^\a} \, \rmd\mu(z) \\
    & = \sum_{\a,\b} p_{\b} \overline{q_{\a}} K(\a, \b) \\
    & =: \langle \hat{p}, \hat{q} \rangle_K,
\end{align*}
so that the inner product on polynomials induced by \(\mu\) is equivalent to the inner product on \(\CC^n\) for sufficiently large \(n\) induced by the square matrix \(K_n = [K(\a',\b')]_{0\preceq \a',\b' \preceq \a}\) where \(\a\) takes shortlex-position \(n\) in \(\NN_0^d\).

Now consider the discussion regarding the (upper-triangular) Cholesky factorisation \(A = F^*F\) and the Gram--Schmidt algorithm in \cite{Con96} preceding Remark 1.6.6. Of particular interest to us is the implication of equation (6.13): the columns of \(F^{-1}\) form an orthonormal basis of the inner product space \((\CC^n, \langle\cdot,\cdot\rangle_{A})\), and, perhaps surprisingly, this basis is the same as that obtained via the normalised Gram--Schmidt algorithm with respect to \(A\). Identifying for each \(k = 0, \ldots, n\) the \(k^{\mathrm{th}}\) basis vector \(e_k\) of \(\CC^n\) with the monomial \(z^\b\), where \(\b \in \NN_0^d\) has shortlex-position \(k\), we obtain an orthonormal basis for the space \(\CC[z_1, \ldots, z_d]_{\a}\) of polynomials of shortlex-highest degree \(\a\).

Notice that we are in precisely the situation discussed therein --- our \(A\) is \(K_n\) --- although as we have already remarked, we no longer have the condition \(A_{ii} = 1, 0\leq i \leq n\). However, this caveat does not affect the main ideas of the argument, rather being a simplification to illustrate the point being discussed, and we may extend the argument to non-normalised kernels as follows: if one has an upper-triangular Cholesky factorisation
\[
    B = F^{\ast}F
\]
with \(B\) normalised, then
\[
    A = \big(\diag(\overline{A_{jj}})F^{\ast}\big) \big(F\diag(A_{jj})\big)
\]
is an upper-triangular Cholesky factorisation of our non-normalised \(A\). We hence retain the connection between the Cholesky factorisation of \(A\) and the Gram--Schmidt algorithm in this setting.

It remains to recall that the inner product on \(\CC^n\) induced by \(A\) and the inner product on polynomials induced by \(\mu\) coincide when viewing elements of \(\CC^n\) as coefficients vectors of elements of \(\CC[z_1, \ldots, z_d]\). The orthonormal polynomials \(\p_{\a}\) we obtained by Gram--Schmidt therefore have coefficient vectors arising from the columns of the Cholesky factorisations of truncations of \(K\). Finally for this discussion, we remark that each of these vectors is in fact the \emph{final} column of the inverse of some upper-triangular Cholesky factor by choosing \(n\) appropriately. Hence by letting \(n\) go to infinity we can obtain \(\p_{\a}\) for each \(\a \in \NN_0^d\) as the polynomial whose coefficient vector is the final column of the inverse of the appropriate Cholesky factor.

\begin{remark}
    This connection is exploited also in \cite{CJ02b}, particularly in the proof of Theorem 3.2, where they refer to vectors \(P_{i,j}\) for \(i,j \in \NN\): \(P_{i,j}\) is the final column of the inverse of the Cholesky factor of the truncated kernel window \([K(n,m)]_{i \leq n,m \leq j}\). Note that this construction occurs before the relevant multi-Toeplitz structure is applied, and therefore applies to any positive kernel, including the ones in which we are interested.
    
    If \(\a\) has position \(n\) in the shortlex ordering, then, we have that \(P_{1,n}\) of \cite{CJ02b} is the coefficient vector of our \(\p_{\a}\). We note that while this construction does give a bi-indexed family \((P_{i,j})_{i,j\in\NN}\), we shall primarily be interested in \((\p_{\a})_{\a\in\NN_0^d}\) and hence in the singly-indexed subfamily \((P_{1,n})_{n\in\NN}\).
\end{remark}

The following result appears in \cite{CJ02b} in the proof of Theorem 3.2 and takes the role of a ``weakening" of the Szeg\H{o} recurrences that is nevertheless sufficient for our application.

\begin{proposition}[\cite{CJ02b}, Proof of Theorem 3.2]
    \label{prop:CJ02bRecurrences}
    Let \(K : \NN_0 \times \NN_0 \to \CC\) be a positive kernel such that \(K(n,n) = 1\) for all \(n\in\NN_0\), and for \(i,j\in\NN\) define the vector \(P_{i,j}\), as in the previous remark, to be the final column of the inverse of the upper-triangular Cholesky factor of the window \([K(n,m)]_{i \leq n,m \leq j}\) of \(K\). Correspondingly, let the vectors \(P_{i,j}^{\#}\) be the \emph{first} column of the \emph{lower}-triangular Cholesky factor of the same matrix.

    Let \((\g_{i,j})_{i,j \in \NN_0}\) be the Verblunsky coefficients of the kernel \(K\) arising from \cite[Theorem 1.5.3]{Con96} and let \(d_{i,j} := \sqrt{1 - \lvert \g_{i,j} \rvert^2}\). Then
    \begin{equation}
        \label{eq:SzegoRecurrence}
        P_{1,n} = \frac{1}{d_{1,n}} \begin{bmatrix}0 \\ P_{2,n} \end{bmatrix} - \frac{\g_{1,n}}{d_{1,n}}\begin{bmatrix} P^{\#}_{1,n-1} \\ 0 \end{bmatrix}
    \end{equation}
    and
    \begin{equation}
        \label{eq:SharpRecurrence}
        P_{1,n}^{\#} = -\frac{\overline{\g_{1,n}}}{d_{1,n}} \begin{bmatrix}0 \\ P_{2,n} \end{bmatrix} + \frac{1}{d_{1,n}}\begin{bmatrix} P^{\#}_{1,n-1} \\ 0 \end{bmatrix}.
    \end{equation}
\end{proposition}

Note that this result arises from \cite{CJ02b}, and as such the condition that \(K(n,n) = 1\) for all \(n \in \NN_0\) is implicit. As we have already discussed, any kernel \(K\) that does not satisfy this condition can be obtained from a kernel \(\tilde{K}\) that does by replacing the Cholesky factors \(F_{i,j}\) of truncations of \(\tilde{K}\) by \(F_{i,j} \cdot \mathrm{diag}\big(K(n,n)\big)_{n=i}^{j}\). As such, we may extend \Cref{prop:CJ02bRecurrences} to kernels of non-trivial positive measures on \(\partial \BB^d\) as follows.

\begin{lemma}
    Let \(K\) be the kernel of a non-trivial positive measure \(\mu\) on \(\partial \BB^d\), considered as a kernel on \(\NN_0\) via \(K(n,m) = K(\a, \b)\) when \(\a, \b\) take respective positions \(n,m\) in \(\NN_0^d\) according to the shortlex ordering. For \(i,j \in \NN_0\), let
    \[
        K_{i,j} := \begin{bmatrix} K(n,m) \end{bmatrix}_{\substack{i\leq n \leq j \\ i \leq m \leq j}}
    \]
    have upper-triangular Cholesky factorisation \(K_{i,j} = F_{i,j}^*F_{i,j}\) and lower-triangular Cholesky factorisation \(K_{i,j} = G_{i,j}^*G_{i,j}\). Let \(P_{i,j}\) be the final column of \(F_{i,j}^{-1}\) and let \(P_{i,j}^{\#}\) be the first column of \(G_{i,j}^{-1}\). Then \eqref{eq:SzegoRecurrence} and \eqref{eq:SharpRecurrence} hold.
\end{lemma}
\begin{proof}
    Define \(\tilde{K}\) by
    \[
        \tilde{K}(n,m) = \frac{K(n,m)}{\sqrt{K(n,n)}\sqrt{K(m,m)}}, \quad \text{ for } n,m \in \NN
    \]
    and for \(i,j \in \NN\), let \(\tilde{F}_{i,j}, \tilde{G}_{i,j}, \tilde{P}_{i,j}\) and \(\tilde{P}_{i,j}^{\#}\) be as above with \(\tilde{K}\) in place of \(K\). Then we have that
    \[
        F_{i,j} = \tilde{F}_{i,j}D_{i,j} \quad \text{ and } \quad G_{i,j} = \tilde{G}_{i,j}D_{i,j}
    \]
    where \(D_{i,j} := \mathrm{diag}\big(K(n,n)\big)_{n = i}^{j} \in \CC^{(1+j-i) \times (1+j-i)}\) is always invertible, since \(\mu\) is non-trivial so that \(K(n,n)\) is never zero, and hence
    \[
        F_{i,j}^{-1} = D^{-1}_{i,j} \tilde{F}_{i,j}^{-1} \quad \text{ and } \quad G_{i,j}^{-1} = D^{-1}_{i,j} \tilde{G}_{i,j}^{-1}.
    \]
    By considering the final column of the former relation and the first column of the latter, we see that
    \[
        P_{i,j} = D^{-1}_{i,j} \tilde{P}_{i,j} \quad \text{ and } \quad P_{i,j}^{\#} = D^{-1}_{i,j} \tilde{P}_{i,j}^{\#},
    \]

    Notice that \(\tilde{K}(n,n) = 1\) for all \(n \in \NN\), so that by \Cref{prop:CJ02bRecurrences} we have that \eqref{eq:SzegoRecurrence} and \eqref{eq:SharpRecurrence} hold with \(\tilde{P}_{i,j}, \tilde{P}_{i,j}^{\#}\) in place of \(P_{i,j}, P_{i,j}^{\#}\). Multiplying the first of these recurrences by \(D_{i,j}^{-1}\), we see that
    
    \begin{align*}
        P_{1,n} & = D_{1,n}^{-1} \tilde{P}_{1,n} = D_{1,n}^{-1} \left(\frac{1}{d_{1,n}} \begin{bmatrix}0 \\ \tilde{P}_{2,n} \end{bmatrix} - \frac{\g_{1,n}}{d_{1,n}}\begin{bmatrix} \tilde{P}^{\#}_{1,n-1} \\ 0 \end{bmatrix} \right) \\
        & = \frac{1}{d_{1,n}} \begin{bmatrix}0 \\ D_{2, n}^{-1} \tilde{P}_{2,n} \end{bmatrix} - \frac{\g_{1,n}}{d_{1,n}}\begin{bmatrix} D_{1,n-1}^{-1}\tilde{P}^{\#}_{1,n-1} \\ 0 \end{bmatrix} \\
        & = \frac{1}{d_{1,n}} \begin{bmatrix}0 \\ P_{2,n} \end{bmatrix} - \frac{\g_{1,n}}{d_{1,n}}\begin{bmatrix} P^{\#}_{1,n-1} \\ 0 \end{bmatrix},
    \end{align*}
    so that \eqref{eq:SzegoRecurrence} holds, and \eqref{eq:SharpRecurrence} is proven similarly with the recurrence for \(\tilde{P}^{\#}_{i,j}\).
\end{proof}

We note that, as in the univariate case, \(P_{1,1} = P_{1,1}^{\#} = K(1,1)\) so that for a probability measure \(\mu\) on \(\partial \BB^d\) we have \(P_{1,1} = P_{1,1}^{\#} = 1\).

\begin{definition}
    Let \(\mu\) be a non-trivial positive measure on \(\partial\BB^d\) and let \(P_{i,j}^{\#}\) be as above. We define the \emph{sharp (orthonormal) polynomials} \(\p_{\a}^{\#}\) for \(\a \in \NN_0^d\) to be the polynomial with coefficient vector \(P_{1,n}^{\#}\) where \(\a\) takes position \(n\) under the shortlex ordering.
\end{definition}

We may rewrite \eqref{eq:SzegoRecurrence} and \eqref{eq:SharpRecurrence} in terms of the orthonormal polynomials as follows.

\begin{theorem}[Szeg\H{o}-Type Recurrences]
    \label{thm:Recurrences}
    Let \(\mu\) be a non-trivial positive measure on \(\partial\BB^d\) and recall for \(\a,\b\in\NN_0^d\) and \(i,j\in\NN\) the polynomials \(\p_\a\), \(\p_\a^{\#}\) and the vectors \(P_{i,j}\), \(P_{i,j}^{\#}\) from the above discussion. For \(\a \in \NN_0^d\), let \(\p_{e_1, \a}\) be the polynomial with coefficients \(P_{2,n}\) where \(\a\) has shortlex position \(n\), and write \(P_{2,n} = [\p_{e_1, \a}^{(\b)}]_{\b}\). Then we have the following analogue of the Szeg\H{o} recurrences:
    \begin{equation}
        \p_\a(z) = \frac{1}{d_{0,\a}} \left(\sum_{\b} \p_{e_1, \a}^{(\b)} z^{\mathrm{succ}(\b)}\right) - \frac{\g_{0,\a}}{d_{0,\a}} \p^{\#}_{\mathrm{prec}(\a)}(z),
    \end{equation}
    \begin{equation}
        \p_{\a}^{\#}(z) = -\frac{\overline{\g_{0,\a}}}{d_{0,\a}} \left(\sum_{\b} \p_{e_1, \a}^{(\b)} z^{\mathrm{succ}(\b)}\right) + \frac{1}{d_{0,\a}} \p_{\mathrm{prec}(\a)}^{\#}(z).
    \end{equation}
\end{theorem}

\begin{remark}
    \label{rem:SharpPolys}
    Note that the expression \(\sum_{\b} \p_{e_1, \a}^{(\b)} z^{\mathrm{succ}(\b)}\) appearing in these recurrences is the result of shifting each monomial of the polynomial \(\p_{e_1, \a}\) up one position in the shortlex ordering, so that this is in one sense a natural analogue of the term \(z\p_{n-1}\) appearing in the classical Szeg\H{o} recurrences.
\end{remark}

\begin{remark}
    Here the sharp polynomials \(\p_{\a}^{\#}\) take the role of the traditional reverse polynomials in the Szeg\H{o} recurrences. In one variable, the reverse polynomial of a polynomial \(p\in\CC[z]\) may be obtained directly via the formula
    \[
        p^{\#}(z) = z^{\mathrm{deg}(p)}\overline{p(1/\overline{z})};
    \]
    however, for polynomials on \(\BB^d\) no such relation holds. Rather, the recurrence relations \eqref{eq:SzegoRecurrence} and \eqref{eq:SharpRecurrence} provide motivation for us to define the sharp orthogonal polynomials as we have, and we do not claim a general notion of ``reverse polynomial'' for an arbitrary polynomial \(p \in \CC[z_1, \ldots, z_d]\).
\end{remark}

Finally, we recall that the association between a measure on the unit circle and its orthogonal polynomials on the unit circle is bijective, as one can recover the measure from the orthogonal polynomials via Verblunsky's theorem (also called Favard's theorem for the unit circle), see \cite[Theorem 1.7.11]{Sim05a}. In the kernel framework of \cite{Con96}, there exists a similar result --- Remark 1.5.4 --- that says a kernel \(K\) is determined uniquely not by the Verblunsky coefficients alone, but by the pair \(((K(n,n))_{n}, (\g_{i,j})_{i,j})\). We now use this to acquire a corresponding result in our setting as follows; note that when \(d = 1\), \(K\) is normalised and one recovers Verblunsky's theorem.

\begin{theorem}[Verblunsky's Theorem on \(\partial\BB^d\)]
    Let \((\g_{\a,\b})_{\a,\b \in \NN_0^d} \subseteq \DD\) and for \(\a\in\NN_0^d\) let $(s_{\alpha})_{\alpha \in \NN_0^d}$ be a given real-valued sequence. Let \(K\) be the unique positive kernel on \(\NN_0^d\) with Verblunsky coefficients \((\g_{\a,\b})_{\a,\b\in\NN_0^d}\) and diagonal \(K(\a,\a) = s_\a\) for all \(\a \in \NN_0^d\) guaranteed by \cite[Remark 1.5.4]{Con96}.
    
    Then there is a unique positive measure \(\mu\) on \(\partial\BB^d\) with Verblunsky coefficients \((\g_{\a,\b})_{\a,\b\in\NN_0^d}\) such that
    \[
        \int_{\partial\BB^d} \lvert \z^\a \rvert^2 \, \rmd\mu(\z) = s_{\alpha}
    \]
    for all \(\a \in \NN_0^d\) if and only if
    \[
        K(\a,\b) = \sum_{j=1}^d K(\a + e_j, \b + e_j)
    \]
    for all \(\a, \b \in \NN_0^d\); in this case, the moment kernel of \(\mu\) is \(K\).
\end{theorem}
\begin{proof}
    The proof follows immediately from \cite[Remark 1.5.4]{Con96} and \Cref{thm:Kernel-Measure}.
\end{proof}

\section{A Multivariate Christoffel Function}

\begin{definition}
    \label{def:CDK}
    Let \(\mu\) be a non-trivial probability measure on \(\partial\BB^d\) with orthonormal polynomials \((\p_\a)_{\a \in \NN_0^d}\). For \(n \in \NN_0\) and \(z,w \in \overline{\BB^d}\), we associate to \(\mu\) the \emph{Christoffel--Darboux kernel}
    \[
        \k_{\mu,n}(z,w) := \sum_{\a = 0}^{\a(n)} \p_{\a}(z) \overline{\p_{\a}(w)} = 1 + \sum_{0 \prec \a \preceq a(n)} \p_{\a}(z) \overline{\p_{\a}(w)},
    \]
    and \emph{Christoffel approximate}
    \[
        \l^{(d)}_{n}(z; \rmd \mu) := \k_{\mu,n}(z,z)^{-1} = \left(\sum_{\a = 0}^{\a(n)} \lvert\p_{\a}(z)\rvert^2\right)^{-1}.
    \]
\end{definition}

For \(\a \in \NN_0^d\), recall that \(\CC[z_1, \ldots, z_d]_{\a} := \mathrm{span}\{z^{\b}\}_{0 \preceq \b \preceq \a} \).

\begin{lemma}
    \label{lem:CFMinimumFormula}
    Let \(\mu\) be a non-trivial probability measure on \(\partial\BB^d\). For \(n\in\NN_0\) and \(z \in \overline{\BB^d}\) we have
    \begin{equation}
        \label{eq:CFMinimumFormula}
        \l^{(d)}_n(z; \rmd \mu) = \min\left\{ \int_{\partial \BB^d} \lvert q(\z) \rvert^2 \, \rmd \mu(\z) : q \in \CC[z_1, \ldots, z_d]_{\a(n)}, q(z) = 1 \right\}.
    \end{equation}
\end{lemma} 
\begin{proof}
    We proceed as in the proof of the analogous result of \cite{GK25}, which in turn is modelled on a proof of \cite{BMV23}, though here things are decidedly simpler: our variables now commute, and as \(k = 1\) the tensor products collapse to multiplication.

    Fix \(z \in \overline{\BB^d}\). For \(\a \in \NN_0^d\), let \(c_{\a} := \l^{(d)}_n(z; \rmd \mu) \overline{\p_{\a}(z)}\) and
    \[
        P_n := \sum_{\a=0}^{\a(n)} c_{\a} \p_{\a}.
    \]

    Then
    \[
        P_n(z) = \sum_{\a=0}^{\a(n)} \l^{(d)}_n(z; \rmd \mu) \overline{\p_{\a}(z)} \p_{\a}(z) =  \l^{(d)}_n(z; \rmd \mu) \sum_{\a=0}^{\a(n)}\lvert\p_{\a}(z)\rvert^2 = 1
    \]
    so \(P_n\) is admissible for our minimisation. Subsequently, for \(\z \in \partial \BB^d\),
    \[
        \lvert P_n(\z) \rvert^2 = \left( \sum_{\a=0}^{\a(n)} c_{\a} \p_{\a}(\z) \right) \left( \overline{\sum_{\a=0}^{\a(n)} c_{\a} \p_{\a}(\z)} \right) = \sum_{\a,\b=0}^{\a(n)} c_{\a}\overline{c_{\b}} \p_{\a}(\z)\overline{\p_{\b}(\z)}
    \]
    so that
    \begin{align*}
        \int_{\partial \BB^d} \lvert P_n(\z) \rvert^2 \, \rmd \mu(\z) & =  \sum_{\a,\b=0}^{\a(n)} c_{\a}\overline{c_{\b}} \int_{\partial\BB^d} \p_{\a}(\z)\overline{\p_{\b}(\z)} \, \rmd \mu(\z) \\
        & = \sum_{\a,\b=0}^{\a(n)} c_{\a}\overline{c_{\b}} \d_{\a, \b} \\
        & = \sum_{\a=0}^{\a(n)} \lvert c_{\a}\rvert^2 \\
        & = \sum_{\a=0}^{\a(n)} \lvert \l^{(d)}_n(z; \rmd \mu) \overline{\p_{\a}(z)}\rvert^2 \\
        & = \l^{(d)}_n(z; \rmd \mu)^2 \left( \sum_{\a=0}^{\a(n)} \lvert \p_{\a}(z) \rvert^2 \right) \\
        & = \l^{(d)}_n(z; \rmd \mu),
    \end{align*}
    so that \(\l^{(d)}_n(z; \rmd \mu)\) is in the set over which we are minimising.

    Now, given an arbitrary
    \[
        q = \sum_{\a = 0}^{\a(n)} a_{\a} \p_{\a} \in \CC[z_1, \ldots, z_d]_{\a(n)}
    \]
    with \(q(z) = 1\), consider the matrix
    \[
        X = \begin{bmatrix}
            a_{0} & \cdots & a_{\a(n)} \\[6pt] \overline{\p_{0}(z)} & \cdots & \overline{\p_{\a(n)}(z)}
        \end{bmatrix}.
    \]
    By \cite[Theorem 14.10]{Pau02}, viewing \(\CC\) as a Hilbert C*-module over itself, we have that
    \begin{align*}
        \begin{bmatrix} \sum_{k = 1}^{n+1} \langle X_{ik}, X_{jk} \rangle \end{bmatrix}_{i,j=1}^{2} & = \begin{bmatrix}
            \sum_{\a = 0}^{\a(n)} a_{\a}\overline{a_{\a}} & \sum_{\a = 0}^{\a(n)} a_{\a}\overline{\overline{\p_{\a}(z)}} \\[6pt]
            \sum_{\a = 0}^{\a(n)} \overline{\p_{\a}(z)}\overline{a_{\a}} & \sum_{\a = 0}^{\a(n)} \overline{\p_{\a}(z)}\overline{\overline{\p_{\a}(z)}}
        \end{bmatrix} \\[6pt]
        & = \begin{bmatrix}
            \sum_{\a = 0}^{\a(n)} \lvert a_{\a} \rvert^2 & q(z) \\[6pt] \overline{q(z)} & \sum_{\a = 0}^{\a(n)} \lvert \p_{\a}(z) \rvert^2
        \end{bmatrix} \\[6pt]
        & = \begin{bmatrix}
            \sum_{\a = 0}^{\a(n)} \lvert a_{\a} \rvert^2 & 1 \\ 1 & \l^{(d)}_n(z; \rmd \mu)^{-1}
        \end{bmatrix}
    \end{align*}
    is a positive matrix; taking the Schur complement then further implies the inequality
    \[
        1\cdot\big(\l^{(d)}_n(z; \rmd \mu)^{-1}\big)^{-1}\cdot 1 \leq \sum_{\a = 0}^{\a(n)} \lvert a_{\a} \rvert^2,
    \]
    i.e.
    \[
        \l^{(d)}_n(z; \rmd \mu) \leq \sum_{\a = 0}^{\a(n)} \lvert a_{\a} \rvert^2.
    \]
    On the other hand, by an identical computation as for \(P_n\),
    \[
        \int_{\partial\BB^d} \lvert q(\z) \rvert^2 \, \rmd \mu(\z) = \sum_{\a,\b=0}^{\a(n)} a_{\a}\overline{a_{\b}} \int_{\partial\BB^d} \p_{\a}(\z)\overline{\p_{\b}(\z)} \, \rmd \mu(\z) = \sum_{\a = 0}^{\a(n)} \lvert a_{\a} \rvert^2.
    \]
    It follows that
    \[
         \l^{(d)}_n(z; \rmd \mu) \leq \int_{\partial\BB^d} \lvert q(\z) \rvert^2 \, \rmd \mu(\z)
    \]
    for any \(q \in \CC[z_1, \ldots, z_d]_{\a(n)}\) such that \(q(z) = 1\) and we are done.    
\end{proof}

We can now define the Christoffel function of a measure on \(\partial\BB^d\).

\begin{definition}
    For a non-trivial probability measure \(\mu\) on \(\partial\BB^d\), we define the \emph{Christoffel function} of \(\mu\) to be the pointwise limit of the Christoffel approximates,
    \[
        \l_{\infty}^{(d)}(z; \rmd \mu) := \lim_{n\to\infty} \l^{(d)}_n(z; \rmd \mu).
    \]
\end{definition}

\begin{lemma}
    \label{lem:CFExists}
    Let \(\mu\) be a non-trivial probability measure on \(\partial\BB^d\). The Christoffel function of \(\mu\) exists for all \(z \in \CC^d\) and is given by
    \[
        \l_{\infty}^{(d)}(z; \rmd \mu) = \inf\left\{ \int_{\partial \BB^d} \lvert q(\z) \rvert^2 \, \rmd \mu(\z) : q \in \CC[z_1, \ldots, z_d], q(z) = 1 \right\}.
    \]
\end{lemma}
    \begin{proof}
    First, \(\p_{0} = 1\), so that \(\l^{(d)}_n(z; \rmd \mu) \geq 0\) for all \(n\in\NN_0\), bounding the sequence below. Recall the polynomials \((P_n)_{n=0}^{\infty}\) solving \eqref{eq:CFMinimumFormula} from the proof of \Cref{lem:CFMinimumFormula}. As \(P_n \in \CC[z_1, \ldots, z_d]_{\a(n)} \subseteq \CC[z_1, \ldots, z_d]_{\a(n+1)}\) for each \(n \in \NN_0\), we have
    \[
        \l^{(d)}_{n+1}(z; \rmd \mu) \leq \l^{(d)}_n(z; \rmd \mu)
    \]
    for each \(n\). Thus \((\l^{(d)}_n(z; \rmd \mu))_{n\in\NN_0}\) forms a decreasing sequence bounded below and so converges; we call this limit \(\l_{\infty}^{(d)}(z; \rmd \mu)\).
    
    Finally, notice that taking this limit provides
    \begin{align*}
        \l_{\infty}^{(d)}(z; \rmd \mu) & = \min_n \l^{(d)}_n(z; \rmd \mu) = \inf_n \l^{(d)}_n(z; \rmd \mu) \\ & = \inf\left\{ \int_{\partial \BB^d} \lvert q(\z) \rvert^2 \, \rmd \mu(\z) : q \in \CC[z_1, \ldots, z_d], q(z) = 1 \right\}
    \end{align*}
    directly.
\end{proof}

In fact, when the measure \(\mu\) is sufficiently well-behaved, we can expand this infimum to one over the space \(H^{\infty}(\BB^d)\) of all bounded holomorphic functions on \(\BB^d\).

\begin{definition}
    Let \(f \in H^{\infty}(\partial \BB^d)\). By Kor\'anyi's Fatou theorem \cite{Kor69}, the \(K\)-limit (see, e.g., \cite[5.4.6]{Rud08}) \(\klim_{z \to \z} f(z)\) exists for \(\rmd\s\)-almost all \(\z \in \partial\BB^d\); when this limit exists, we shall denote it by simply \(f(\z)\).
\end{definition}

\begin{theorem}
    \label{thm:HInftyEntropy}
    Let \(\mu\) be a non-trivial probability measure on \(\partial\BB^d\) and suppose \(\mu\) is absolutely continuous with respect to \(\s\). Then we can expand the domain of our minimisation problem for \(\l_{\infty}^{(d)}(\cdot; \rmd\mu)\) to
    \begin{equation}
        \label{eq:HInftyEntropy}
        \l_{\infty}^{(d)}(z; \rmd \mu) = \inf\left\{\int_{\partial \BB^d} \lvert f(\z) \rvert^2 \, \rmd\mu(\z) \; : \; f \in H^{\infty}(\BB^d), \; f(z) = 1\right\}.
    \end{equation}
\end{theorem}
\begin{proof}
    Let \(f \in H^{\infty}(\BB^d)\); since \(\mu\) is absolutely continuous with respect to \(\s\), it has Lebesgue decomposition \(\rmd\mu(\z) = w(\z) \rmd \s(\z)\) for some \(w \in L^1(\partial\BB^d)\), and moreover, since \(\klim_{z \to \z} f(z)\) exists \(\rms \s\)-a.e. it also exists \(\rmd\mu\)-a.e. so that we may refer to \(f(\z)\) for \(\z \in \partial\BB^d\) up to a \(\mu\)-null set. Now, use \emph{integration by slices} (see \cite[Proposition 1.4.7(i)]{Rud08}) to calculate that
    \begin{align*}
        \lim_{r \to 1} \int_{\partial\BB^d} \lvert f(r\z) \rvert^2 \, \rmd\mu(\z) & = \lim_{r \to 1} \int_{\partial\BB^d} \lvert f(r\z) \rvert^2 w(\z) \,\rmd\s(\z) \\
        & = \lim_{r\to1} \int_{\partial\BB^d} \rmd\s(\z) \frac{1}{2\pi} \int_0^{2\pi} \lvert f(r e^{i\th} \z) \rvert^2 w(e^{i\th}\z) \,\rmd\th.
    \end{align*}
    By the maximum modulus principle \cite[Theorem 10.24]{Rud87} applied to the univariate function \(z \mapsto f(z\z)\),
    \[
        \frac{1}{2\pi} \int_0^{2\pi} \lvert f(r e^{i\th} \z) \rvert^2 w(e^{i\th}\z) \,\rmd\th \leq \frac{1}{2\pi} \int_0^{2\pi} \lvert f(e^{i\th} \z) \rvert^2 w(e^{i\th}\z) \,\rmd\th
    \]
    for any \(\z \in \partial\BB^d\), so that by the dominated convergence theorem \cite[Theorem 1.34]{Rud87} we may exchange the limit and the integral to obtain
    \begin{align*}
        \lim_{r\to1} \int_{\partial\BB^d} \rmd\s(\z) \frac{1}{2\pi} \int_0^{2\pi} \lvert f(r e^{i\th} \z) \rvert^2 w(e^{i\th}\z) \,\rmd\th & = \int_{\partial\BB^d} \rmd\s(\z) \lim_{r\to1} \frac{1}{2\pi} \int_0^{2\pi} \lvert f(r e^{i\th} \z) \rvert^2 w(e^{i\th}\z) \,\rmd\th.
    \end{align*}
    By the univariate theory \cite[Proposition 2.5.3]{Sim05a}
    \[
         \lim_{r\to1} \frac{1}{2\pi} \int_0^{2\pi} \lvert f(r e^{i\th} \z) \rvert^2 w(e^{i\th}\z) \,\rmd\th = \frac{1}{2\pi} \int_0^{2\pi} \lvert f(e^{i\th} \z) \rvert^2 w(e^{i\th}\z) \,\rmd\th,
    \]
    so that finally we arrive at
    \[
        \lim_{r \to 1} \int_{\partial\BB^d} \lvert f(r\z) \rvert^2 \, \rmd\mu(\z) = \int_{\partial\BB^d} \lvert f(\z) \rvert^2 \, \rmd\mu(\z).
    \]

    All that remains is to see that, as \(f \in H^\infty(\BB^d)\), it has a Taylor expansion which is uniformly convergent inside \(\BB^d\). For \(r < 1\) and \(\z \in \partial\BB^d\), then, \(r\z \in \BB^d\) so that the function \(\z \mapsto f(r\z)\) can be uniformly approximated by the Taylor polynomials of \(f\). Hence expanding the domain of our minimisation problem does not decrease the value of the problem's solution.
\end{proof}

\section{A Multivariate Szeg\H{o} Entropy}

To prove a multivariate Szeg\H{o}--Verblunsky theorem, we broadly follow the strategy of the proof in \cite{Sim05a} using the Poisson kernel, specialised to the case \(z = 0\). Throughout, a measure \(\mu\) shall have Lebesgue decomposition \(\rmd\mu(\z) = w(\z)\rmd\s(\z) + \rmd\mu_\rms(\z)\), where \(w \in L^1(\partial\BB^d)\), \(\s\) is the rotation-invariant Lebesgue measure on \(\partial\BB^d\) and \(\mu_\rms\) is singular with respect to \(\s\).

\begin{theorem}
    \label{thm:BoundedAbove}
	Let \(\mu\) be a non-trivial probability measure on \(\partial \BB^d\) with Lebesgue decomposition \(\rmd \mu(\z) = w(\z)\rmd \s(\z) + \rmd \mu_{\rms}(\z)\). Then we have
	\begin{equation}
        \label{eq:BoundedAbove}
		\l_\infty^{(d)}(0; \rmd \mu) \geq \exp\left(\int_{\partial \BB^d} \log(w(\zeta)) \, \rmd \s(\zeta)\right).
	\end{equation}
\end{theorem}
\begin{proof}
	Let \(q \in \CC[z_1, \ldots, z_d]\) be a polynomial such that \(q(0) = 1\). We have
	\begin{align*}
		\int_{\partial \BB^d} \lvert q(\z) \rvert^2 \, \rmd \mu & \geq \int_{\partial \BB^d} \lvert q(\z) \rvert^2 \cdot w(\z) \, \rmd \s(\z) \\
		& = \int_{\partial \BB^d}\exp\left(\log\left(\lvert q(\z) \rvert^2\right) + \log\left(w(\z)\right)\right)\,\rmd\s \\
        & \geq \exp\left(\int_{\partial \BB^d} 2 \log \lvert q(\z) \rvert + \log\left(w(\z)\right) \, \rmd \s\right)
	\end{align*}
    by Jensen's inequality \cite[Theorem 3.3]{Rud87} with respect to the probability measure \(\s\). Then integration by slices allows us to rewrite the latter integral as
    \[
        \int_{\partial \BB^d} 2 \log \lvert q(\z) \rvert + \log\left(w(\z)\right) \, \rmd \s = \int_{\partial \BB^d} \, \rmd\s(\z) \int_{0}^{2\pi} 2\log\lvert q_{\z}(e^{i\th}) \rvert + \log(w_{\z}(e^{i\th})) \, \frac{\rmd \th}{2\pi},
    \]
    where \(q_\z(z) := q(z\z) \in \CC[z]\) and \(w_\z(e^{i\th}) := w(e^{i\th}\z) \in L^1(\TT)\) are the univariate \emph{slice functions} associated to \(\z \in \partial \BB^d\) (see \cite[Section 1.2.5]{Rud08}).

    Now, since \(q_\z\) is holomorphic, \(\log \lvert q_\z \rvert\) is harmonic. If \(q_\z\) has a zero inside the unit disc, then the integral of \(\log\lvert q_\z\rvert\) diverges to \(-\infty\), so that the right hand side of \Cref{eq:BoundedAbove} is simply 0; as we have seen, \(\p_0 \equiv 1\) so that the left hand side is non-negative, so in this instance the inequality holds trivially. Otherwise \(q_\z\) has no zeros inside the unit disc, in which case the mean value property of harmonic functions on \(\overline{\DD}\) provides
    \[
        \int_0^{2\pi} \log \lvert q_\z(e^{i\th}) \rvert \, \frac{\rmd \th}{2\pi} = \log \lvert q_\z(0) \rvert = \log \lvert q(0,\ldots,0) \rvert = \log \lvert 1 \rvert = 0.
    \]
    
    This simplifies our integral to
    \[
        \int_{\partial \BB^d} 2 \log \lvert q(\z) \rvert + \log\left(w(\z)\right) \, \rmd \s = \int_{\partial \BB^d} \, \rmd\s(\z) \int_{0}^{2\pi}\log(w_{\z}(e^{i\th})) \, \frac{\rmd \th}{2\pi}
    \]
    and undoing the integration by slices we obtain the inequality
    \[
        \int_{\partial \BB^d} \lvert q(\z) \rvert^2 \, \rmd \mu \geq \exp\left(\int_{\partial \BB^d} \log(w(\z)) \, \rmd \s\right).
    \]
    
    Finally, the infimum of the left hand side over all polynomials \(q\) satisfying \(q(0) = 1\) must also satisfy this inequality, and by \Cref{lem:CFExists} this infimum is precisely \(\l^{(d)}_{\infty}(0; \rmd\mu)\).
\end{proof}

\begin{lemma}
    \label{lem:SchurFunctionLimits}
    Let \(\mu\) be a discrete positive measure on \(\partial \BB^d\). There exists \(G : \BB^d \to \overline{\DD}\) such that
    \begin{enumerate}[label=\textnormal{(\alph*)}]
        \item \(\lvert G(z) \rvert \leq \norm{z}\) for \(z \in \BB^d\);
        \item The radial limit \(\lim_{r\to1} G(r\z)\) exists for \(\rmd \s\)-almost all \(\z \in \partial \BB^d\), and is different from 1;
        \item The radial limit \(\lim_{r \to 1} G(r\z)\) exists and is equal to 1 for \(\rmd\mu\)-almost all \(\z \in \partial \BB^d\).
    \end{enumerate}
\end{lemma}
\begin{remark}
     We shall appeal shortly to a particular multivariate analogy of the Herglotz transform \(F\) of \(\mu\); such an \(F\) will be holomorphic with positive real part everywhere on \(\BB^d\). We remark that, as Herglotz' representation theorem no longer holds on \(\BB^d\), functions of this form comprise only one possible multivariate abstraction of the Herglotz class on the disc; several candidates, including the one we consider, were discussed in \cite{MP05} and further investigated in \cite{Jur10}.
\end{remark}
\begin{proof}[Proof of \Cref{lem:SchurFunctionLimits}]
    Define \(F : \BB^d \to \CC\) by
    \[
        F(z) := \int_{\partial \BB^d} \frac{1 + \langle z,\z \rangle}{1 - \langle z,\z \rangle} \, \rmd \mu(\z).
    \]
    As \(\mu\) is discrete, we may write
    \[
        \mu = \sum_{k = 1}^{\infty} \r_k \d_{\z_k}
    \]
    where \(\r_k > 0\), \(\sum_k \r_k = \mu(\partial \BB^d) = 1\) and the \(\z_k \in \partial \BB^d\) are distinct, and then \(F\) becomes
    \[
        F(z) = \sum_{k = 1}^{\infty} \r_k \frac{1 + \langle z,\z_k \rangle}{1 - \langle z,\z_k \rangle}.
    \]
    It is readily seen that \(\Re \, F \geq 0\), the Cayley transform of \(F\) is a function \(G : \BB^d \to \overline{\DD}\), i.e. let \(G\) be defined for \(z \in \BB^d\) by
    \[
        \frac{1 + G(z)}{1 - G(z)} = F(z).
    \]
    This \(G\) shall be our candidate.
    \begin{enumerate}[label=\textnormal{(\alph*)}]
        \item Since \(F(0) = \sum_k \r_k = 1\), we have \(G(0) = 0\). By a generalisation of the Schwarz lemma \cite[Theorem 8.1.2]{Rud08} it follows that \(\lvert G(z) \rvert \leq \norm{z}\) for \(z \in \BB^d\).
        
        \item This would follow from Theorem 9.3.2 of \cite{Rud08} if we knew that \(\lvert G(z) \rvert < 1\) for \(z \in \BB^d\). Notice that \(G(z) = - \frac{1 - F(z)}{1 + F(z)}\), so that
        \[
            \lvert G(z) \rvert^2 = \left \lvert \frac{1 - F(z)}{1 + F(z)} \right\rvert^2 = \frac{1 - 2\Re F(z) + \lvert F(z) \rvert^2}{1 + 2 \Re F(z) + \lvert F(z) \rvert^2},
        \]
        and hence \(\lvert G(z) \rvert = 1\) if and only if \(\Re F(z) = 0\).

        We next calculate that
        \[
            \Re F(z) = \frac12 (F(z) + \overline{F(z)}) = \sum_{k=1}^\infty \r_k \frac{1 - \lvert \langle z,\z_k\rangle \rvert^2}{\lvert 1 - \langle z, \z_k \rangle \rvert^2}
        \]
        and observe that each term in this sum is non-negative for \(z \in \BB^d\). Thus if \(\Re F(z) = 0\) for some \(z \in \BB^d\), then we would have
        \[
            \frac{1 - \lvert \langle z,\z_k\rangle \rvert^2}{\lvert 1 - \langle z, \z_k \rangle \rvert^2} = 0, \quad k = 1, 2, \ldots,
        \]
        i.e. \(\lvert \langle z, \z_k \rangle \rvert = 1\) for all \(k\); however, this cannot happen since by Cauchy--Schwarz
        \[
            \lvert \langle z, \z_k \rangle \rvert^2 \leq \norm{z} \norm{\z_k} < 1.
        \]
        Thus \(\Re F(z) > 0\) for \(z \in \BB^d\) and accordingly \(\lvert G(z) \rvert < 1\) for \(z \in \BB^d\).
        
        By \cite[Theorem 9.3.2]{Rud08} we now have that for \(\rmd\s\)-almost all \(\z \in \partial \BB^d\), \(\lim_{r\to1} G(r\z)\) exists and is different from \(1\).   
        
        \item We computed above for \(z \in \BB^d\) that
        \begin{align*}
             F(z) = \sum_{k = 1}^{\infty} \r_k \frac{1 + \langle z,\z_k \rangle}{1 - \langle z,\z_k \rangle},
        \end{align*}
        so for \(r < 1\) and \(\z \in \partial \BB^d\) we have
        \[
            \lvert F(r\z) \rvert = \left\lvert\sum_{k = 1}^{\infty} \r_k \frac{1 + r\langle \z,\z_k \rangle}{1 - r\langle \z,\z_k \rangle}\right\rvert \leq \sum_{k = 1}^{\infty} \r_k \left\lvert\frac{1 + r\langle \z,\z_k \rangle}{1 - r\langle \z,\z_k \rangle}\right\rvert.
        \]

        Discreteness of \(\mu\) means that \(\z \in \supp(\mu)\) if and only if \(\z = \z_k\) for some \(k\in\NN\) and therefore that \(\mu(\{\z\}) = \r_k > 0\); hence ``for \(\rmd\mu\)-almost all \(\z \in \partial\BB^d\)" here means ``for \(\z_k\) for all \(k\in\NN\)".

        With this in mind, suppose \(\z = \z_n\) for some \(n \in \NN\). Notice first \(\langle \z_n, \z_k \rangle = 1\) if and only if \(k = n\), in which case
        \begin{align*}
           \left\lvert\frac{1 + r\langle \z,\z_k \rangle}{1 - r\langle \z,\z_k \rangle}\right\rvert = \left\lvert\frac{1 + r\langle \z_n,\z_n \rangle}{1 - r\langle \z_n,\z_n \rangle}\right\rvert = \frac{1 + r}{1 - r},
        \end{align*}
        and this clearly diverges as \(r \to 1\).
        
        It follows that for \(\rmd\mu\)-almost all \(\z \in \partial \BB^d\), there exists \(n \in \NN\) such that \(\z = \z_n\) and accordingly that
        \begin{align*}
            \lim_{r \to 1} \lvert F(r\z) \rvert & \leq \lim_{r\to1} \sum_{k = 1}^{\infty} \r_k \left\lvert\frac{1 + r\langle \z,\z_k \rangle}{1 - r\langle \z,\z_k \rangle}\right\rvert \\
            & = \lim_{r\to1} \left( \r_n \left\lvert\frac{1 + r\langle \z_n,\z_n \rangle}{1 - r\langle \z_n,\z_n \rangle}\right\rvert + \sum_{\substack{k=1 \\ k \neq n}}^{\infty} \r_k \left\lvert\frac{1 + r\langle \z_n,\z_k \rangle}{1 - r\langle \z_n,\z_k \rangle}\right\rvert\right) \\
            & = \lim_{r\to1} \left( \r_n \left(\frac{1+r}{1-r}\right) + \sum_{\substack{k=1 \\ k \neq n}}^{\infty} \r_k \left\lvert\frac{1 + r\langle \z_n,\z_k \rangle}{1 - r\langle \z_n,\z_k \rangle}\right\rvert\right),
        \end{align*}
        which must diverge as each term in the sum is positive.

        Now the radial limit \(\lim_{r\to1} \lvert F(r\z) \rvert \) diverges for \(\rmd\mu\)-a.e. \(\z \in \partial \BB^d\). Since \(\frac{1+G(r\z)}{1-G(r\z)} = F(r\z)\) and by (a) \(\lvert 1 + G(r\z) \rvert \leq 1 + \norm{r\z} < 2\), the only way for \(F(r\z)\) to diverge as \(r \to 1\) is for \(\lvert 1 - G(r\z) \rvert \to 0\), i.e. \(G(r\z) \to 1\). Since \(F(r\z)\) diverges \(\rmd\mu\)-a.e., we have that \(G(r\z) \to 1\) \(\rmd\mu\)-a.e., as claimed.
    \end{enumerate}
\end{proof}

\begin{remark}
     In the univariate theory, the analogous result holds for any singular (with respect to Lebesgue) measure, of which the discrete measures are a subset. This can be seen by combining well-known results (e.g. in \cite{Rud87}) on convergence of the symmetric derivative of the measure \(\mu\) and the Poisson transform \(P[\mu]\) with the fact that \(P[\mu](e^{i\th}) = \Re F(e^{i\th})\) for \(\rmd \mu\)-almost all \(e^{i\th} \in \TT\), as argued in \cite{Sim05a}.

     On the other hand, analogous results on the measure and Poisson transform exist when \(d > 1\), e.g. in \cite{Rud08}. However, the relationship between \(P[\mu]\) and \(\Re F\) is much less well-defined: one is initially only able to obtain
     \begin{equation}
        \label{eq:MultivariatePoissonInequality}
        \Re F(z) \geq \int_{\partial \BB^d} P(z,\z)^{\frac{1}{d}} \, \rmd \mu(\z)
     \end{equation}
     for \(P\) the Poisson kernel of \cite{Rud08}, via Jensen's inequality (for concave functions). 
\end{remark}

\begin{proposition}
    \label{prop:PropertiesOfPj}
    Let \(\mu\) be a discrete positive measure on \(\partial \BB^d\). There exists a sequence \((p_j)_{j\in\NN}\) of polynomials such that:
    \begin{enumerate}[label=\textnormal{(\roman*)}]
        \item \(\sup_{\substack{z \in \overline{\BB^d}\\j \in \NN}} \lvert p_j(z) \rvert = 1\);
        \item \(\lim_{j\to \infty} \lvert p_j(\z) \rvert = 0\) for \(\rmd \s \)-almost all \(\z \in \partial \BB^d\);
        \item \(\lim_{j\to\infty} \lvert p_j(z) \rvert = 0\) uniformly for \(z\) in any compact subset of \(\BB^d\);
        \item \(\lim_{j\to \infty} p_j(\z) = 1\) for \(\rmd \mu \)-almost all \(\z \in \partial \BB^d\).
    \end{enumerate}
\end{proposition}
\begin{proof}
    Let \(G\) be defined as in \Cref{lem:SchurFunctionLimits} and let \(H(z) = \frac12 (1 + G(z))\); notice that holomorphy of \(G\) implies that of \(H\). Items (a)-(c) of \Cref{lem:SchurFunctionLimits} respectively imply:
    \begin{enumerate}[label=\textnormal{(\alph*)}]
        \item with the triangle inequality, \(\lvert H(z) \rvert \leq \frac12 (1 + \norm{z})\);
        \item \(\lim_{r\to1} \lvert H(r\z) \rvert < 1\) for \(\rmd\s\)-a.e. \(\z\in\partial\BB^d\);
        \item \(\lim_{r \to 1} H(r\z) = 1\) for \(\rmd\mu\)-a.e. \(\z\in\partial\BB^d\);
    \end{enumerate}
    where the inequality in (b) is strict because \(\lim_{r\to1}G(r\z) \neq 1\).

    Now, as \cite[1.2.6]{Rud08} remarks, holomorphic functions on a ball have global power series representations on that ball; for \(r < 1\), then, \(z \mapsto H(rz)\) has a global power series representation on the larger ball \(\frac1r \BB^d\). This power series converges uniformly strictly inside this ball, so in particular for \(z \in \overline{\BB^d} \subsetneq \frac1r \BB^d\) we can uniformly approximate \(H(rz)\) by polynomials by truncating a power series.

    With this in mind, let \(r_n = 1 - \frac1n\) for \(n\in\NN\). Then there exist polynomials \(q_n\) such that for \(n\in\NN\) and all \(z \in \BB^d\),
    \[
        \lvert q_n(z) - H(r_n z) \rvert < \frac{1}{3n},
    \]
    and applying the triangle inequality to \(\lvert q_n(z) - H(r_n z) + H(r_nz)\rvert\) we obtain the bound
    \[
        \lvert q_n(z) \rvert < 1 - \frac1{6n}, \quad \text{ for \(z \in \BB^d\)}.
    \]
    By continuity, this bound extends to all \(z \in \overline{\BB^d}\) at the cost of the strict inequality, so by taking a limit, for \(z \in \overline{\BB^d}\) we have \(\lim_{n\to\infty} \lvert q_n(z) \rvert \leq 1\). Furthermore, if this limit were equal to 1 for all \(z\) in a set of positive \(\s\)-measure, then we would have that \(\lim_{n\to\infty} \lvert H(r_nz) \rvert = 1\) for all \(z\) in this set, contradicting (b). Our inequality is therefore strict \(\rmd\s\)-a.e.:
    \[
        \lim_{n\to\infty} \lvert q_n(\z) \rvert < 1 \quad \text{ for \(\rmd\s\)-a.e. \(\z \in \partial {\BB^d}\)}.
    \]
    Similarly then the supremum over \(n\) is bounded as
    \[
        \sup_{n\in\NN} \lvert q_n(\z) \rvert < 1 \quad \text{ for \(\rmd\s\)-a.e. \(\z \in \partial {\BB^d}\)}.
    \]
    On the other hand, for \(\rmd\mu\)-a.e. \(\z \in \partial\BB^d\),
    \[
        \lim_{n\to\infty} q_n(\z) = \lim_{n\to\infty} H(r_n \z) = \lim_{r\to1} H(r\z) = 1,
    \]
    and in turn we have
    \[
        \limsup_{n\in\NN} \lvert 1 - q_n(\z) \rvert = 0 \quad \text{ for \(\rmd\mu\)-a.e. \(z \in \overline{\BB^d}\)}.
    \]
    The set where this property does not hold is thus \(\mu\)-measure zero, so
    \[
        \mu\left\{\z \in \partial \BB^d : \limsup_{n\in\NN} \lvert 1 - q_n(\z) \rvert = 0 \right\} = \mu(\partial \BB^d) = 1.
    \]
    We follow the procedure described in the remark succeeding \cite[Theorem 2.5.1]{Sim05a} to obtain a sequence \((k_n)_{n\in\NN}\) and for each \(\z \in \partial \BB^d\) an integer \(j_0(\z)\) such that, for \(\rmd\mu\)-a.e. \(\z \in \partial \BB^d\),
    \[
        \lvert 1 - q_n(\z) \rvert \leq \frac{1}{j^2} \quad \text{ when \(n \geq k_j\) and \(j \geq j_0(\z)\)}.
    \]

    For \(j \in \NN\), we now define \(p_j := q_{k_j}^j\). We complete the proof of the result as follows:
    \begin{enumerate}[label=\textnormal{(\roman*)}]
        \item For \(z \in \BB^d\) and hence for \(z \in \overline{\BB^d}\), \(\lvert p_j(z) \rvert \leq (1 - \frac{1}{6k_j})^j \leq 1\), so taking suprema,
        \[
            \sup_{\substack{z \in \overline{\BB^d}\\j \in \NN}} \lvert p_j(z) \rvert \leq 1.
        \]
        On the other hand,
        \[
            \lvert 1 - p_n(\z) \rvert = \lvert 1^n - q_{k_n}(\z)^n \rvert = \lvert (1 - q_{k_n}(\z))r(\z) \rvert = \lvert 1 - q_{k_n}(\z) \rvert \lvert r(z) \rvert
        \]
        where \(r\) is some polynomial (which is therefore bounded on the ball). Since for \(\rmd\mu\)-a.e. \(\z\) we saw that \(\lim_{n\to\infty} q_n(\z) = 1\), it follows that \(\lim_{n\to\infty}p_n(\z) = 1\) also for such \(\z\), and therefore
        \[
            \sup_{\substack{z \in \overline{\BB^d}\\j \in \NN}} \lvert p_j(z) \rvert \geq 1,
        \]
        showing equality.

        \item For \(\rmd\s\)-a.e. \(\z \in \partial \BB^d\), we have \(\sup_{n\in\NN} \lvert q_n(\z) \rvert < 1\), so for \(\rmd\s\)-a.e. \(\z\),
        \[
            \lim_{j\to\infty}\lvert p_j(\z) \rvert = \lim_{j\to\infty} \lvert q_j(\z)^j\rvert \leq \lim_{j\to\infty} \left(\sup_{n\in\NN}\lvert q_n(\z) \rvert\right)^j = 0.
        \]

        \item Let \(K \subseteq \BB^d\) be compact. Then \(\max\{\norm{z} : z \in K\}\) is attained by some \(z_K \in K\), and this maximum is less than \(1\); hence for \(z \in K\) we have
        \[
            \lvert H(z) \rvert \leq \frac12 (1 + \norm{z}) \leq \frac12(1 + \norm{z_K}) < 1.
        \]
        As \(\lvert q_n(z) \rvert \leq 1 - \frac{1}{3n}\) for each \(n\in\NN\) and \(\lim_{n\to\infty} q_n(z) = H(z)\),
        \[
            \sup_{\substack{z \in K\\n \in \NN}} \lvert q_n(z) \rvert < 1,
        \]
        so by the same argument as for (ii) we have that
        \[
            \lim_{j\to\infty} \lvert p_j(z) \rvert = 0.
        \]
        Moreover, as \(K\) is compact, the supremum over \(K\) is achieved and therefore the convergence is uniform.

        \item Calculate directly that
        \begin{align*}
            \lvert 1 - p_j(\z) \rvert & = \lvert 1 - q_{k_j}(\z)^j \rvert \\
            & = \lvert 1 - q_{k_j}(\z) + q_{k_j}(\z) - q_{k_j}(\z)^j \rvert \\
            & \leq \lvert 1 - q_{k_j}(\z) \rvert +\lvert q_{k_j}(\z) - q_{k_j}(\z)^j \rvert \\
            & = \lvert 1 - q_{k_j}(\z) \rvert +\lvert q_{k_j}(\z) \rvert \lvert 1 - q_{k_j}(\z)^{j-1} \rvert \\
            & \leq \lvert 1 - q_{k_j}(\z) \rvert + \lvert 1 - q_{k_j}(\z)^{j-1} \rvert
        \end{align*}
        and then repeat this process \(j-1\) times to get
        \[
             \lvert 1 - p_j(\z) \rvert \leq j\cdot \lvert 1 - q_{k_j}(\z) \rvert.
        \]
        By definition of the sequence \((k_j)_{j\in\NN}\), taking \(n = k_j\) we know that
        \[
            \lvert 1 - q_{k_j}(\z) \rvert \leq \frac{1}{j^2}
        \]
        for sufficiently large \(j\). Take the limit to finally see that
        \[
            \lvert 1 - p_n(\z) \rvert \leq \frac{1}{n} \to 0.
        \]
    \end{enumerate}
\end{proof} 

We shall now apply these results to see that, when the singular part of a measure is discrete, its Christoffel function is determined solely by its absolutely continuous part.

\begin{theorem}
    \label{thm:AbsolutelyContinuousDeterminesEntropy}
    Let \(\mu = \mu_{\ac} + \mu_\rms\) be a the Lebesgue decomposition of a non-trivial probability measure \(\mu\) on \(\partial \BB^d\) with \(\mu_\rms\) discrete. Then for \(z_0 \in \BB^d\),
    \[
        \l_{\infty}^{(d)}(z_0; \rmd\mu) =\l_{\infty}^{(d)}(z_0; \rmd\mu_{\rm{ac}}).
    \]
\end{theorem}
\begin{proof}
    Fix \(z_0 \in \BB^d\). For any polynomial \(q \in \CC[z_1, \ldots, z_d]\) we have that
    \[
        \int_{\BB^d} \lvert q(\z) \rvert^2 \, \rmd\mu(\z) = \int_{\BB^d} \lvert q(\z) \rvert^2 \, \rmd\mu_{\rm{ac}}(\z) + \int_{\BB^d} \lvert q(\z) \rvert^2 \, \rmd\mu_{\rms}(\z) \geq \int_{\BB^d} \lvert q(\z) \rvert^2 \, \rmd\mu_{\ac}(\z),
    \]
    so taking infima over all \(q\) such that \(q(z_0) = 1\) we have that
    \[
        \l_{\infty}^{(d)}(z_0; \rmd\mu) \geq \l_{\infty}^{(d)}(z_0; \rmd\mu_{\ac}).
    \]
    
    Let \(q \in \CC[z_1,\ldots,z_d]\) be a polynomial such that \(q(z_0) = 1\), and let \((p_j)_{j\in\NN}\) be as in \Cref{prop:PropertiesOfPj}. Define a new sequence of polynomials \((q_j)_{j\in\NN}\) by
    \[
        q_j(z) := q(z)\frac{1-p_j(z)}{1-p_j(z_0)}
    \]
    and observe that \(q_j(z_0) = 1\) for all \(j\in\NN\).
    
    Since \(z_0 \in \BB^d\), it sits inside some compact subset of the unit ball; by \Cref{thm:AbsolutelyContinuousDeterminesEntropy}(iii) we therefore have that \(\lim_{j\to\infty} p_j(z_0) = 0\).
    By \Cref{thm:AbsolutelyContinuousDeterminesEntropy}(ii) we have \(\lim_{j\to\infty} p_j(\z) = 0\) for \(\rmd\s\)-almost all \(\z \in \partial \BB^d\) and so also for \(\rmd\mu_{\ac}\)-almost all \(\z \in \partial\BB^d\), while \Cref{thm:AbsolutelyContinuousDeterminesEntropy}(iv) states the same limit is 1 for \(\rmd\mu_\rms\)-almost all \(\z \in \partial \BB^d\).

    Therefore, by compactness of \(\partial \BB^d\), we may interchange the following integral with the limit as \(j \to \infty\):
    \begin{align*}
        \int_{\partial \BB^d} \lvert q_j(\z) \rvert^2 \, \rmd\mu(\z) & = \int_{\partial\BB^d} \lvert q(\z) \rvert^2 \left\lvert \frac{1-p_j(\z)}{1-p_j(z_0)} \right\rvert^2 \, \rmd\mu_{\rm{ac}}(\z) + \int_{\partial \BB^d} \lvert q(\z) \rvert^2 \left\lvert\frac{1-p_j(\z)}{1-p_j(z_0)}\right\rvert^2 \, \rmd\mu_{\rms}(\z) \\
        & \to \int_{\partial\BB^d} \lvert q(\z) \rvert^2 \left\lvert \frac{1-0}{1-0} \right\rvert^2 \, \rmd\mu_{\rm{ac}}(\z) + \int_{\partial\BB^d} \lvert q(\z) \rvert^2 \left\lvert \frac{1-1}{1-0}\right\rvert^2 \, \rmd\mu_{\rms}(\z) \\
        & = \int_{\partial \BB^d} \lvert q(\z) \rvert^2 \, \rmd\mu_{\ac}(\z).
    \end{align*}

    Recall that the \(q_j\) are polynomials that sent \(z_0\) to 1; hence \(\l_{\infty}^{(d)}(z_0; \rmd\mu)\) is bounded above by the left hand side for each \(j\) and hence by their limit, that is,
    \[
        \l_{\infty}^{(d)}(z_0; \rmd\mu) \leq \lim_{j\to\infty} \int_{\partial \BB^d} q_j(\z) \, \rmd\mu(\z) = \int_{\partial \BB^d} q(\z) \, \rmd\mu_{\ac}(\z)
    \]
    for any polynomial \(q\) with \(q(z_0) = 1\). Taking an infimum over all such polynomials,
    \begin{multline*}
        \l_{\infty}^{(d)}(z_0; \rmd\mu) = \inf\{\l_{\infty}^{(d)}(z_0; \rmd\mu) : q \in \CC[z_1, \ldots, z_d], q(z_0) = 1\} \\ \leq \inf\left\{\int_{\partial \BB^d} q(\z) \, \rmd\mu_{\ac}(\z) : q \in \CC[z_1, \ldots, z_d], q(z_0) = 1\right\} = \l_{\infty}^{(d)}(z_0; \rmd\mu_{\ac}).
    \end{multline*}
    This shows equality.
\end{proof}

We are now in a position to complement \Cref{thm:BoundedAbove} with a converse inequality when we place additional conditions upon the measure.

\begin{lemma}
    \label{lem:BoundedBelow}
    Let \(\mu\) be a non-trivial probability measure on \(\partial \BB^d\) with Lebesgue decomposition \(\rmd \mu(\z) = w(\z)\rmd \s(\z) + \rmd \mu_{\rms}(\z)\) such that \(\mu_\rms\) is discrete. If there exists \(f \in H^{\infty}(\BB^d)\) such that \(f(0) = 1\) and
    \begin{equation}
        \label{eq:SVHypothesis}
        \int_{\partial \BB^d} \lvert f(\z) \rvert^2 w(\z) \, \rmd\s(\z) \leq \exp \left( \int_{\partial \BB^d} \log(w(\z)) \, \rmd\s(\z) \right),
    \end{equation}
    then
    \begin{equation}
        \label{eq:BoundedBelow}
        \l_{\infty}^{(d)}(0; \rmd \mu) \leq \exp\left(\int_{\partial \BB^d} \log(w(\z)) \, \rmd \s(\z) \right).
    \end{equation}
\end{lemma}
\begin{proof}
    Let \(f \in H^{\infty}(\BB^d)\) be such that \(f(0) = 1\) and suppose that \eqref{eq:SVHypothesis} holds. Since \(f\) is admissible for \eqref{eq:HInftyEntropy}, it follows from \Cref{thm:AbsolutelyContinuousDeterminesEntropy} that
    \[
        \l^{(d)}_{\infty}(0; \rmd \mu) = \l^{(d)}_{\infty}(0; w \, \rmd \s) \leq \int_{\partial \BB^d} \lvert f(\z) \rvert^2 w(\z) \, \rmd\s(\z),
    \]
    and the result follows.
\end{proof}

\begin{remark}
    When \(d = 1\), \cite{Sim05a} is able to explicitly construct a function \(f \in H^{\infty}(\DD)\) satisfying the hypothesis of \Cref{lem:BoundedBelow} for any \(w \in L^1(\TT)\): essentially, one takes
    \[
        q(z) := \exp\left(\frac{-1}{2} \int_{\TT} \frac{1 + ze^{-i\th}}{1 - ze^{-i\th}} \log(w(e^{i\th})) \, \frac{\rmd \th}{2\pi}\right)
    \]
    and
    \[
        f(z) := \frac{q(z)}{q(0)},
    \]
    though some careful approximation is necessary as in general \(w\) may be badly behaved.
    
    The desired inequality then follows from properties of the Poisson kernel, which arises in one variable as twice the real part of the Herglotz kernel \(H(z,w) = \frac{1 + z\overline{w}}{1 - z\overline{w}}\) and hence appears in \(\lvert q(z) \rvert^2\). For \(d > 1\), this relationship between the Poisson kernel and the choice of Herglotz kernel in which we are interested fails, engendering the need for our additional hypothesis.
\end{remark}

\begin{corollary}
    \label{cor:MultivariateSzego}
    Let \(\mu\) be a non-trivial probability measure on \(\partial \BB^d\) with Lebesgue decomposition \(\rmd \mu(\z) = w(\z)\rmd \s(\z) + \rmd \mu_{\rms}(\z)\), where \(\mu_\rms\) is discrete, and suppose that there exists \(f \in H^{\infty}(\BB^d)\) such that \(f(0) = 1\) and
    \[
        \int_{\partial \BB^d} \lvert f(\z) \rvert^2 w(\z) \, \rmd\s(\z) \leq \exp \left( \int_{\partial \BB^d} \log(w(\z)) \, \rmd\s(\z) \right).
    \]
    Then
    \begin{equation}
        \l_{\infty}^{(d)}(0; \rmd \mu) = \exp\left(\int_{\partial \BB^d} \log(w(\z)) \, \rmd \s(\z) \right).
    \end{equation}
\end{corollary}
\begin{proof}
    This is the combination of \Cref{thm:BoundedAbove} and \Cref{lem:BoundedBelow}.
\end{proof}

In the final section of the paper we shall highlight a number of classes of examples satisfying \eqref{eq:SVHypothesis}.

\section{A Multivariate Summary Theorem}

In this section we use the orthogonal polynomials and Verblunsky coefficients of a measure on \(\partial\BB^d\) that were introduced and studied in Section 2 and Section 3, respectively, to obtain a multivariate analogue of \Cref{thm:SummaryThm}. Throughout, \(\mu\) shall be a non-trivial probability measure on \(\partial\BB^d\) with moment kernel \(K\), Verblunsky coefficients \((\g_{\a,\b})_{\a,\b\in\NN_0^d}\), monic orthogonal polynomials \((\P_\a)_{\a\in\NN_0^d}\), orthonormal polynomials \((\p_\a)_{\a\in\NN_0^d}\) with leading coefficients \((a_{\a,\a})_{\a\in\NN_0^d}\), sharp polynomials \((\p_\a^{\#})_{\a\in\NN_0^d}\) and Christoffel function \(\l^{(d)}_\infty(z; \rmd\mu)\).

\begin{lemma}
    \label{i=ii}
    For \(\a \in \NN_0^d\), we have that
    \[
        \norm{\P_{\a}} = a_{\a, \a}^{-1}.
    \]
\end{lemma}
\begin{proof}
    The orthonormal and monic orthogonal polynomials corresponding to \(\mu\) are related by
    \[
        \p_{\a} = \frac{\P_\a}{\norm{\P_\a}_\mu}
    \]
    and as \(\P_{\a}\) is monic, the leading coefficient of \(\p_{\a}\) is \(a_{\a,\a} = \frac{1}{\norm{\P_\a}_\mu}\).
\end{proof}

\begin{definition}
    Let \(\mu\) be a probability measure on \(\partial\BB^d\) and define the \emph{moment matrix determinants} of \(\mu\) to be
    \[
        D_{\a} = \det [K(\a', \b')]_{0 \preceq \a',\b' \preceq \a}.
    \]
    for \(\a \in \NN_0^d\). We note that when $d=1$ the matrix $ [K(\a', \b')]_{0 \preceq \a',\b' \preceq \a}$ is Toeplitz. However, this need not hold when $d > 1$. 
\end{definition}

To prove our next equality, we proceed as in the proof of \cite[Theorem 3.1]{CJ02b}.

\begin{lemma}
    \label{ii=iii}
    For \(\a \in \NN_0^d\), we have that
    \[
        a_{\a,\a}^{-2} = \frac{D_{\a}}{D_{\mathrm{prec}(\a)}}.
    \]
\end{lemma}
\begin{proof}
    Since \(\mu\) is non-trivial, \(D_\a\) is nonzero for all \(\a \in \NN_0^d\). It follows from Gram--Schmidt that
    \[
        \langle \p_{\a}, z^{\b} \rangle_{\mu} = 0
    \]
    for \(0 \preceq \b \prec \a\), and hence by the definition of \(\langle \cdot, \cdot \rangle_{\mu}\),
    \[
        \sum_{0 \preceq \b \preceq \a} a_{\a,\b} K(\b', \b) = 0
    \]
    for all \(0 \preceq \b' \prec \a\).
    
    The Cramer rule for the linear system
    \begin{align*}
        & \sum_{0 \preceq \b \preceq \a} a_{\a,\b} K(\b', \b) = 0, \quad 0 \preceq \b' \prec n, \\
        & \sum_{0 \preceq \b \preceq \a} a_{\a,\b} z^{\b} = \p_\a
    \end{align*}
    tells us that
    \[
        a_{\a,\a} = \frac{\p_{\a} D_{\mathrm{prec}(\a)}}{\det \begin{bmatrix}
            [K(\a',\b')]_{0 \preceq \a',\b' \preceq \a} \\
            \begin{array}{cccc}1 & z^{e_1} & \cdots & z^{\a} \end{array}
        \end{bmatrix}},
    \]
    which we rearrange to obtain
    \[
        \p_{\a} = \frac{a_{\a,\a}}{D_{\mathrm{prec}({\a})}}\det \begin{bmatrix}
            [K(\a',\b')]_{0 \preceq \a',\b' \preceq \a} \\
            \begin{array}{cccc}1 & z^{e_1} & \cdots & z^{\a} \end{array}
        \end{bmatrix}.
    \]
    By the relation \(\langle z^{\a}, z^{\b} \rangle_{\mu} = K(\a,\b)\), expanding the determinant along the bottom row yields
    \[
        \left\langle \det \begin{bmatrix}
            [K(\a',\b')]_{0 \preceq \a',\b' \preceq \a} \\
            \begin{array}{cccc}1 & z^{e_1} & \cdots & z^{\a} \end{array}
        \end{bmatrix}, z^{\a} \right\rangle_{\mu} = D_{\a}.
    \]
    
    On the other hand,
    \[
        z^{\a} = \frac{1}{a_{\a,\a}} \p_{\a} + \sum_{0\preceq \b \prec \a} c_{\b} z^{\b} 
    \]
    for some coefficients \((c_{\b})_{\b=0}^{\mathrm{prec}(\a)}\), so by orthogonality of \(\p_{\a}\) and \(z^{\b}\),
    \begin{align*}
        D_{\a} & = \left\langle \frac{D_{\mathrm{prec}(\a)}}{ a_{\a,\a}}\p_{\a},  \frac{1}{a_{\a,\a}} \p_{\a} + \sum_{0\preceq \b \prec \a} c_{\b} z^{\b} \right\rangle = \frac{D_{\mathrm{prec}(\a)}}{ a_{\a,\a}^2} \langle \p_{\a}. \p_{\a} \rangle = \frac{D_{\mathrm{prec}(\a)}}{ a_{\a,\a}^2},
    \end{align*}
    so that finally
    \[
        a_{\a,\a}^{-2} = \frac{D_{\a}}{D_{\mathrm{prec}(\a)}}.
    \]
\end{proof}

\begin{lemma}
    \label{iii=iv}
    For \(\a \in \NN_0^d\), we have that
    \[
        \frac{D_{\a}}{D_{\mathrm{prec}(\a)}} = K(\a, \a) \cdot \prod_{0 \preceq \b \preceq \mathrm{prec}(\a)}(1 - \lvert \g_{\b,\a} \rvert^2).
    \]
\end{lemma}
\begin{proof}
    By Theorem 1.5.10 of \cite{Con96}, we have
    \[
        D_{\a} = \prod_{\b = 0}^{\a} K(\b, \b) \cdot \prod_{0 \preceq \b' \prec \a' \preceq \a} d^2_{\b', \a'},
    \]
    recalling that \(d_{\b,\a} := \sqrt{1 - \lvert \g_{\b,\a} \rvert^2}\).
    
    This allows us to extend \Cref{ii=iii} as follows:
    \begin{align*}
        \frac{D_{\a}}{D_{\mathrm{prec}(\a)}} & = \frac{\prod_{\b = 0}^{\a} K(\b, \b) \cdot \prod_{0 \preceq \b' \prec \a' \preceq \a} d^2_{\b', \a'}}{\prod_{\b = 0}^{\mathrm{prec}(\a)} K(\b, \b) \cdot \prod_{0 \preceq \b' \prec \a' \preceq \mathrm{prec}(\a)} d^2_{\b', \a'}} \\
        & = K(\a, \a) \cdot \frac{\prod_{0 \preceq \b' \prec \a' \preceq \a} d^2_{\b', \a'}}{\prod_{0 \preceq \b' \prec \a' \preceq \mathrm{prec}(\a)} d^2_{\b', \a'}} \\
        & = K(\a, \a) \cdot \prod_{0 \preceq \b \preceq \mathrm{prec}(\a)}(1 - \lvert \g_{\b,\a} \rvert^2).
    \end{align*}
\end{proof}

\begin{remark} 
    Note that the right hand side here has the form of a product of defects of Verblunsky coefficients, with the additional factor of \(K(\a, \a)\), and recall that in the univariate and noncommutative settings this factor is identically one. We therefore recover equality of analogues of all four items of \cite[Theorem 6.3(i)-(iv)]{GK25}.
\end{remark}

\begin{lemma}
    \label{vi=v=vii}
    For \(\a \in \NN_0^d\), we have that
    \[
        \lvert \p_{\a}^{\#}(0) \rvert^{-2} = \prod_{\b=0}^{\a} (1 - \lvert \g_{0,\b}\rvert^2) = \left(\sum_{\b = 0}^{\a} \lvert \p_{\b}(0) \rvert^2\right)^{-1}.
    \]
\end{lemma}
\begin{proof}
    Recall the Szeg\H{o}-type recurrences \eqref{eq:SzegoRecurrence} and \eqref{eq:SharpRecurrence}. Note that \(\sum_{\b} \p_{e_1, \a}^{(\b)} z^{\mathrm{succ}(\b)}\) has zero constant term. We may therefore see inductively for \(\a \in \NN_0^d\) that \eqref{eq:SharpRecurrence} implies
    \begin{align}
            \p_\a^{\#}(0) & = - \frac{\overline{\g_{0,\a}}}{d_{0,\a}} \cdot 0 + \frac{1}{d_{0,\a}}\p_{\mathrm{prec}(\a)}^{\#}(0) \nonumber \\
            & = \frac{1}{d_{0,\a}} \cdot \frac{1}{d_{0,\mathrm{prec}(\a)}} \p_{\mathrm{prec}^2(\a)}^{\#}(0) \nonumber \\
            & = \prod_{\b = 0}^{\a} \frac{1}{d_{0,\a}},       \label{eq:revPolyAtZero}
    \end{align}
    and hence
    \[
        \lvert \p_{\a}^{\#}(0) \rvert^{-2} = \prod_{\b=0}^\a d_{0,\b}^2 = \prod_{\b=0}^{\a} (1 - \lvert \g_{0,\b}\rvert^2).
    \]

    Next, using \eqref{eq:revPolyAtZero}, \eqref{eq:SzegoRecurrence} implies that 
    \[
        \p_\a(0) = \frac{1}{d_{0,\a}} \cdot 0 - \frac{\g_{0,\a}}{d_{0,\a}} \p_{\mathrm{prec}(\a)}^{\#}(0) = -\g_{0,\a} \prod_{\b=0}^{\a} \frac{1}{d_{0,\a}}.
    \]
    Suppose for some \(\a \in \NN_0^d\) that
    \[
        \sum_{\b=0}^{\a} \lvert \p_{\b}(0) \rvert^2 = \prod_{\b=0}^{\a} (1 - \lvert \g_{0,\b}\rvert^2)^{-1}.
    \]
    Then
    \begin{align*}
        \sum_{\b=0}^{\mathrm{succ}(\a)} \lvert \p_{\b}(0) \rvert^2 & = \lvert \p_{\mathrm{succ}(\a)}(0) \rvert^2 + \sum_{\b=0}^{\a} \lvert \p_{\b}(0) \rvert^2 \\
        & = \left\lvert -\g_{0,\mathrm{succ}(\a)} \prod_{\b=0}^{\mathrm{succ}(\a)} \frac{1}{d_{0,\a}} \right\rvert^2 + \prod_{\b=0}^{\a} (1 - \lvert \g_{0,\b}\rvert^2)^{-1} \\
        & = \lvert \g_{0,\mathrm{succ}(\a)} \rvert^2 \prod_{\b=0}^{\mathrm{succ}(\a)} \frac{1}{1 - \lvert \g_{0,\b} \rvert^2} + \prod_{\b=0}^{\a} (1 - \lvert \g_{0,\b}\rvert^2)^{-1} \\
        & = \left(\frac{\lvert \g_{0,\mathrm{succ}(\a)} \rvert^2}{1 - \lvert \g_{0,\mathrm{succ}(\a)} \rvert^2} + 1\right) \prod_{\b=0}^{\a} (1 - \lvert \g_{0,\b}\rvert^2)^{-1} \\
        & = \left(\frac{\lvert \g_{0,\mathrm{succ}(\a)} \rvert^2 + 1 - \lvert \g_{0,\mathrm{succ}(\a)} \rvert^2}{1 - \lvert \g_{0,\mathrm{succ}(\a)} \rvert^2} \right)\prod_{\b=0}^{\a} (1 - \lvert \g_{0,\b}\rvert^2)^{-1} \\
        & = \prod_{\b=0}^{\mathrm{succ}(\a)} (1 - \lvert \g_{0,\b}\rvert^2)^{-1}.
    \end{align*}
    
    Since we have \(\lvert \p_{0}(0) \rvert^2 = 1^2 = \frac{1}{d_{0,0}}\) as a base case, by induction over the shortlex ordering on \(\NN_0^d\),
    \[
        \sum_{\b=0}^{\a} \lvert \p_{\b}(0) \rvert^2 = \prod_{\b=0}^{\a} (1 - \lvert \g_{0,\b}\rvert^2)^{-1}
    \]
    for \(\a \in \NN_0^d\), and we are done.
\end{proof}

\begin{corollary}
    \label{viii=v}
    For \(n \in \NN\),
    \begin{equation}
        \label{eq:ChristoffelApproxEquality}
        \l^{(d)}_{n}(0; \rmd \mu) = \prod_{\a=0}^{\a(n)} (1 - \lvert \g_{0,\a}\rvert^2),
    \end{equation}
    and hence
    \[
        \l_{\infty}^{(d)}(0;\rmd\mu) = \prod_{\a \in \NN_0^d} (1 - \lvert \g_{0,\a}\rvert^2).
    \]
\end{corollary}
\begin{proof}
    The first equality follows immediately from \Cref{vi=v=vii} and \Cref{def:CDK} together; the second follows from the first and \Cref{lem:CFExists}.
\end{proof}

\begin{lemma}
    \label{v=ix}
    For \(n \in \NN\), if the singular part \(\mu_{\rms}\) of \(\mu\) is discrete, we have that
    \[
        \prod_{\a \in \NN_0^d} (1 - \lvert \g_{0, \a} \rvert^2) = \lim_{n \to \infty} \int_{\partial \BB^d} \lvert P_n(\z)\rvert^2 \, \rmd\mu_\ac
    \]
    for some sequence of polynomials \((P_n)_{n\in\NN_0}\) with \(P_n \in \CC[z_1, \ldots, z_d]_{\a(n)}\) for \(n \in \NN\).
\end{lemma}
\begin{proof}
    By the previous corollary and \Cref{thm:AbsolutelyContinuousDeterminesEntropy},
    \[
        \prod_{\a \in \NN_0^d} (1 - \lvert \g_{0, \a} \rvert^2) = \l_{\infty}^{(d)}(z; \rmd \mu) = \l_{\infty}^{(d)}(z; \rmd \mu_{\mathrm{ac}}) = \lim_{n\to\infty} \l_n^{(d)}(z; \rmd \mu_\ac).
    \]
    Recall from the proof of \Cref{lem:CFMinimumFormula}, replacing \(\mu\) by \(\mu_\ac\), that for each \(z \in \overline{\BB^d}\) there exists a family \((P_n)_{n\in\NN_0}\) of polynomials solving \eqref{eq:CFMinimumFormula}, that is, such that \(P_n(z) = 1\) and 
    \[
        \l_n^{(d)}(z; \rmd \mu_\ac) = \int_{\partial \BB^d} \lvert P_n(\z)\rvert^2 \, \rmd\mu_\ac;
    \]
    the result follows as the combination of these equalities.
\end{proof}

We are now in a position to state a multivariate analogue of \Cref{thm:SummaryThm}.

\begin{theorem}
    \label{thm:MainResult}
    Let \(\mu\) be a non-trivial probability measure on \(\partial\BB^d\) with moment kernel \(K : \NN_0^d \times \NN_0^d \to \CC\), orthonormal and monic orthogonal polynomials \((\p_\a)_{\a\in\NN_0^d}\) and \((\P_\a)_{\a\in\NN_0^d}\), Verblunsky coefficients \((\g_{\a,\b})_{\a,\b \in \NN_0^d}\), moment matrix determinants \((D_{\a})_{\a \in \NN_0^d}\), Christoffel approximates \(\l_n^{(d)}(\cdot; \rmd\mu)\) and Christoffel function \(\l_\infty^{(d)}(\cdot; \rmd\mu)\). Let the leading coefficient of \(\p_\a\) be \(a_{\a,\a}\) and let \((P_n)_{n\in\NN_0}\) be the polynomials solving the minimisation problem \eqref{eq:CFMinimumFormula} for \(z = 0\).
    
    The following are equal:
    \begin{enumerate}
        \item[\rm{(i)}] \(\lim_{n\to\infty}\norm{\P_{\a(n)}}^2\);
        \item[\rm{(ii)}] \(\lim_{n\to\infty}a_{\a(n), \a(n)}^{-2}\);
        \item[\rm{(iii)}] \(\lim_{n\to\infty} \frac{D_{\a(n)}}{D_{\mathrm{prec}(\a(n))}}\);
        \item[\rm{(iv)}] \(\lim_{n\to\infty} K(\a(n), \a(n)) \cdot \prod_{\b=0}^{\mathrm{prec}(\a(n))}(1 - \lvert \g_{\b,\a(n)} \rvert^2)\).
    \end{enumerate}
    Moreover, the following are equal, and when \(d > 1\), distinct from the above:
    \begin{enumerate}
        \item[\rm{(v)}] \(\prod_{\a\in\NN_0^d)} (1 - \lvert \g_{0,\a}\rvert^2)\);
        \item[\rm{(vi)}] \(\lim_{n\to\infty}\lvert \p_{\a(n)}^{\#}(0) \rvert^{-2}\);
        \item[\rm{(vii)}] \(\lim_{n\to\infty}\left(\sum_{\a = 0}^{\a(n)} \lvert \p_{\a}(0) \rvert^2\right)^{-1}\);
        \item[\rm{(viii)}] \(\l^{(d)}_{\infty}(0; \rmd \mu)\).
    \end{enumerate}
    When the singular part of \(\mu\) is discrete, we additionally have that \(\mathrm{(v)}\) is equal to
    \begin{enumerate}
        \item[\rm{(viii)'}] \(\l^{(d)}_{\infty}(0; \rmd \mu_{\ac})\);
        \item[\rm{(ix)}] \(\lim_{n\to\infty}\int_{\partial \BB^d} \lvert P_n(\z)\rvert^2 \, \rmd\mu_\ac\).
    \end{enumerate}
    Finally, when additionally there exists \(f \in H^{\infty}(\BB^d)\) such that \(f(0) = 1\) and such that \eqref{eq:SVHypothesis} holds, \(\mathrm{(v)}\) is further equal to
    \begin{enumerate}
        \item[\rm{(x)}] \(\exp\left(\int_{\partial \BB^d} \log(w(\z)) \, \rmd \s(\z) \right)\).
    \end{enumerate}
\end{theorem}

\begin{proof}
    The equality (i) = (ii) follows from \Cref{i=ii}, that (ii) = (iii) follows from \Cref{ii=iii}, and (iii) = (iv) by \Cref{iii=iv}, completing the first list. 
    
    The equalities (v) = (vi) and (v) = (vii) follow from in \Cref{vi=v=vii}, while (v) = (viii) arises by taking a limit in \Cref{viii=v}. When \(\mu_\rms\) is discrete, we have (viii) = (viii)' by \Cref{thm:AbsolutelyContinuousDeterminesEntropy} and (v) = (ix) was proved in \Cref{v=ix}.

    Finally, in the presence of the final assumption in the statement of the theorem, we have (v) = (x) by \Cref{cor:MultivariateSzego}.
\end{proof}

We shall discuss a number of dlasses of measures satisfying the final hypothesis of \Cref{thm:MainResult} in the next section.

\begin{corollary}
    In particular, \Cref{thm:MainResult} implies the following two analogues of classical results when we have some additional hypotheses: firstly,
    \[
        \lim_{n\to\infty} \frac{D_{\a(n)}}{D_{\mathrm{prec}(\a(n))}} = \lim_{n\to\infty} K(\a(n), \a(n)) \cdot \prod_{0 \preceq \b \preceq \mathrm{prec}(\a(n))}(1 - \lvert \g_{\b,\a(n)} \rvert^2).
    \]
    This is analogous to the weak Szeg\H{o} limit theorem, though without a closed form for the limit on the right-hand side.
    
    Secondly, when the singular part of \(\mu\) is discrete and there exists \(f \in H^{\infty}(\BB^d)\) with \(f(0) = 1\) and
    \[
        \int_{\partial \BB^d} \lvert f(\z) \rvert^2 w(\z) \, \rmd\s(\z) \leq \exp \left( \int_{\partial \BB^d} \log(w(\z)) \, \rmd\s(\z) \right),
    \]
    we have
    \[
        \prod_{\a \in \NN_0^d} (1 - \lvert \g_{0,\a}\rvert^2) = \exp\left(\int_{\partial \BB^d} \log(w(\z)) \, \rmd \s(\z) \right)
    \]
    directly analogous to the Szeg\H{o}--Verblunsky theorem. Furthermore, this has the anticipated consequence that
    \[
        \sum_{\a \in \NN_0^d} \lvert \g_{0,\a} \rvert^2 < \infty \quad \text{if and only if} \quad \int_{\partial \BB^d} \log(w(\z)) \, \rmd \s(\z) > -\infty,
    \]
    again characterising log-integrability of the Radon--Nikodym derivative for such measures.
\end{corollary}

\begin{remark}
    While, classically, the weak Szeg\H{o} limit theorem and the Szeg\H{o}--Verblunsky theorem provide different expressions for the same quantity, namely \(\prod_{n\in\NN}(1-\lvert \g_n \rvert^2)\), it is interesting to note that in our setting, the analogues of these quantities that we have obtained are distinct.
\end{remark}

\section{Classes of Measures Which Admit a Szeg\H{o}--Verblunsky Theorem}
\label{sec:ExampleClasses}
 
In this final section we construct a number of classes of measures for which there exists \(f \in H^{\infty}(\BB^d)\) with \(f(0) = 1\) and that satisfy \eqref{eq:SVHypothesis}, that is, classes of measures \(\rmd \mu(\z) = w(\z) \rmd\s(\z) + \rmd\mu_\rms\) with \(\mu_\rms\) discrete that by \Cref{cor:MultivariateSzego} admit a multivariate Szeg\H{o}--Verblunsky theorem. Moreover, we discuss a simple example not satisfying this hypothesis and show that, for this example, the multivariate Szeg\H{o}--Verblunsky theorem does not hold.

\begin{theorem}
    \label{thm:examplesinfty}
    Let \(g \in H^{\infty}(\BB^d)\) be such that \(\norm{g}_{\infty} \leq 1\) and \(1/g \in H^{\infty}(\BB^d)\) and for \(z \in \BB^d\) let \(w(z) = \lvert g(z) \rvert^2\). Let \(\mu_\rms\) be a discrete measure such that \(\rmd\mu(\z) := w(\z)\rmd\s(\z) + \rmd\mu_\rms(\z)\) is a probability measure and suppose that $(\gamma_{\a,\b})_{\a,\b\in \NN_0^d}$ are the Verblunsky coefficients corresponding to $\mu$. Then
    \[
        \prod_{\a \in \NN_0^d} (1 - \lvert \g_{0,\a}\rvert^2) = \exp\left(\int_{\partial \BB^d} \log(w(\z)) \, \rmd \s(\z) \right).
    \]
\end{theorem}
\begin{proof}
    Extend \(w\) to \(\partial \BB^d\) \(\rmd\s\)-almost everywhere by
    \[
        w(\z) := \klim_{z\to\z} w(z) = \lvert \klim_{z\to\z} g(z) \rvert^2,
    \]
    where the latter exists for \(\rmd\s\)-almost all \(\z\) once more by Kor\'anyi's Fatou theorem. First observe that
    \[
        \int_{\partial\BB^d} w(\z) \, \rmd\s(\z) = \int_{\partial \BB^d} \lvert g(\z) \rvert^2 \,\rmd\s(\z) \leq \int_{\partial\BB^d} \norm{g}_\infty^2 \, \rmd\s(\z) \leq \int_{\partial\BB^d} \,\rmd\s = 1 
    \]
    so that \(w \in L^1(\rmd\s)\) with integral bounded above by 1. Since \(\int_{\partial\BB^d} \,\rmd\mu \geq \int_{\partial\BB^d} w \,\rmd\s\), \(w\) may therefore be the Radon--Nikodym derivative of a probability measure \(\mu\).
    
    Let \(f(z) = \frac{g(0)}{g(z)} \in H^{\infty}(\BB^d)\) and observe that \(f(0) = 1\), so that \(f\) is admissible for \eqref{eq:HInftyEntropy} with \(z = 0\). Moreover,
    \[
        \int_{\partial\BB^d} \lvert f(\z) \rvert^2 w(\z) \, \rmd\s(\z) = \lvert g(0) \rvert^2 \int_{\partial\BB^d} \frac{1}{\lvert g(\z) \rvert^2} w(\z) \, \rmd\s(\z) = \lvert g(0) \rvert^2 \int_{\partial\BB^d} \, \rmd\s(\z) = \lvert g(0) \rvert^2 = w(0),
    \]
    and since \(\frac1g \in H^{\infty}(\BB^d)\), we have that \(w(0) > 0\).
    
    Now, by \cite[Proposition 1.5.4]{Rud08}, \(\log w(\cdot) = \log \lvert g(\cdot) \rvert^2\) is subharmonic inside \(\BB^d\); by \cite[Definition 1.5.3(1)]{Rud08} with \(a = 0\) and \(0 < r < 1\) we thence obtain
    \[
        \log w(0) \leq \int_{\partial \BB^d} \log(w(r\z)) \, \rmd\s(\z).
    \]
    (Recall that the mean value property of harmonic functions provides \emph{equality} here when \(d = 1\).) By the maximum modulus principle on the ball \cite[Theorem 10.24]{Rud87}, \(\lvert g(r\z) \rvert \leq \lvert g(\z) \rvert\) for all \(0 < r < 1\) and \(\z \in \partial\BB^d\); the same is true for \(\log w\) since \(t \mapsto t^2\) and \(t \mapsto \log(t)\) are increasing functions on the positive reals. By the dominated convergence theorem \cite[Theorem 1.34]{Rud87} we may take the limit as \(r \to 1\) to see that
    \[
        \log w(0) \leq \int_{\partial \BB^d} \log(w(\z)) \, \rmd\s(\z).
    \]

    By exponentiating, then, we have found an \(f \in H^{\infty}(\BB^d)\) such that \(f(0) = 1\) and
    \[
        \int_{\partial\BB^d} \lvert f(\z) \rvert^2 w(\z) \, \rmd\s(\z) = w(0) \leq \exp \left(\int_{\partial \BB^d} \log(w(\z)) \, \rmd\s(\z)\right),
    \]
    so invoking \Cref{thm:MainResult}, (v) = (x) completes the argument.
\end{proof}

In fact, this argument can be slightly adjusted to allow \(g \in H^p(\BB^d)\) for any \(p \in [1,\infty]\).

\begin{theorem}
    \label{thm:examplesp}
    Fix \(p\in [1,\infty)\) and let \(g \in H^{p}(\BB^d)\) be such that \(\norm{g}_{p} \leq 1\) and \(1/g \in H^{\infty}(\BB^d)\). For \(z \in \BB^d\) let \(w(z) = \lvert g(z) \rvert^p\). Let \(\mu_\rms\) be a discrete measure such that \(\rmd\mu(\z) := w(\z) \rmd\s(\z) + \rmd\mu_\rms(\z)\) is a probability measure and suppose that $(\gamma_{\a,\b})_{\a,\b\in \NN_0^d}$ are the Verblunsky coefficients corresponding to $\mu$. Then
    \[
        \prod_{\a \in \NN_0^d} (1 - \lvert \g_{0,\a}\rvert^2) = \exp\left(\int_{\partial \BB^d} \log(w(\z)) \, \rmd \s(\z) \right).
    \]
\end{theorem}
\begin{proof}
    Once again extend \(g\) and hence \(w\) to \(\partial \BB^d\) \(\rmd\s\)-almost everywhere by \(K\)-limits. First,
    \[
        \int_{\partial\BB^d} w(\z) \, \rmd\s(\z) = \int_{\partial \BB^d} \lvert g(\z) \rvert^p \,\rmd\s(\z) = \norm{g}_p^p \leq 1
    \]
    so as before \(w\) is a valid Radon--Nikodym derivative of a probability measure.
    
    Let \(f(z) = \left(\frac{g(0)}{g(z)}\right)^{\frac{p}{2}}\). Since \(1/g \in H^{\infty}(\BB^d)\), so too is \(1/g^{\frac{p}{2}}\) and hence \(f\). Again \(f(0) = 1\) so \(f\) is admissible for \eqref{eq:HInftyEntropy} with \(z = 0\), and we have that
    \[
        \int_{\partial\BB^d} \lvert f(\z) \rvert^2 w(\z) \, \rmd\s(\z) = \lvert g(0) \rvert^p \int_{\partial\BB^d} \frac{1}{\lvert g(\z) \rvert^p} w(\z) \, \rmd\s(\z) = \lvert g(0) \rvert^p \int_{\partial\BB^d} \, \rmd\s(\z) = \lvert g(0) \rvert^p = w(0).
    \]
    Once more appealing to \cite[Proposition 1.5.4]{Rud08}, the function \(\log w(\cdot) = \log \lvert g(\cdot) \rvert^p\) is subharmonic in \(\BB^d\) and so as above we may see that
    \[
        \log w(0) \leq \int_{\partial \BB^d} \log(w(r\z)) \, \rmd\s(\z).
    \]
    By the maximum modulus principle on the ball \cite[Theorem 10.24]{Rud87} we again have \(\lvert g(r\z) \rvert \leq \lvert g(\z) \rvert\) for all \(0 < r < 1\) and \(\z \in \partial\BB^d\). As \(p \geq 1\), \(t \mapsto t^p\) and \(t \mapsto \log(t)\) are increasing functions on the positive reals, and hence so too is \(\log w\). The dominated convergence theorem \cite[Theorem 1.34]{Rud87} once more allows us to take the limit as \(r \to 1\) to see that
    \[
        \log w(0) \leq \int_{\partial \BB^d} \log(w(\z)) \, \rmd\s(\z).
    \]

    We therefore see that this choice of \(f\) has \(f \in H^{\infty}(\BB^d)\), \(f(0) = 1\) and
    \[
        \int_{\partial\BB^d} \lvert f(\z) \rvert^2 w(\z) \, \rmd\s(\z) = w(0) \leq \exp \left(\int_{\partial \BB^d} \log(w(\z)) \, \rmd\s(\z)\right),
    \]
    so with another invocation of \Cref{thm:MainResult}, (v) = (x) we are done.
\end{proof}

We next use an approximation argument to weaken \Cref{thm:examplesinfty}'s assumption that \(g\) is invertible in \(H^{\infty}(\BB^d)\), obtaining a broader class of examples as follows.

\begin{theorem}
    \label{thm:examplesinftyapprox}
    Suppose that \(g \in H^{\infty}(\BB^d)\) with \(\norm{g}_{\infty} \leq 1\) and \(\Re( g(z) ) \geq 0\) for \(z \in \BB^d\), and suppose that that there exists \(\e > 0\) such that \(w(\z) := \lvert g(\z) \rvert^2 \geq \e\) for \(\rmd \s\)-a.e. \(\z \in \partial\BB^d\). Let \(\mu_\rms\) be a discrete measure such that \(\rmd \mu(\z) = w(\z)\rmd\s + \rmd\mu_\rms(\z)\) is a probability measure on $\partial \BB^d$ and let $(\gamma_{\a,\b})_{\a,\b\in \NN_0^d}$ be the Verblunsky coefficients corresponding to $\mu$. Then
    \[
        \prod_{\a \in \NN_0^d} (1 - \lvert \g_{0,\a}\rvert^2) = \exp\left(\int_{\partial \BB^d} \log(w(\z)) \, \rmd \s(\z) \right) > 0.
    \]
\end{theorem}
\begin{proof}
    We have that \(w\) is a valid Radon--Nikodym derivative for a probability measure by the same argument as for \Cref{thm:examplesinfty}. For \(n \in \NN\), let \(g_n(z) = g(z) + \frac1n\) and \(w_n(\z) = \lvert g_n(\z) \rvert^2\).
    
    Firstly, \(w_n(\z) > w(\z)\) for all \(n \in \NN\) and \(\z \in \partial\BB^d\), so for any \(p \in \CC[z_1, \ldots, z_d]\) with \(p(z) = 1\) and \(n \in \NN\) we have that
    \[
        \int_{\partial\BB^d} \lvert p(\z) \rvert^2 w_n(\z) \, \rmd\s(\z) \geq \int_{\partial\BB^d} \lvert p(\z) \rvert^2 w(\z) \, \rmd\s(\z)
    \]
    and therefore, by taking an infimum over all such \(p\),
    \[
        \l^{(d)}_{\infty}(z; w_n \rmd\s) \geq \l^{(d)}_{\infty}(z; w \rmd\s).
    \]
    Take a limit over \(n\) to see that
    \[
        \lim_{n\to\infty} \l^{(d)}_{\infty}(z; w_n \rmd\s) \geq \l^{(d)}_{\infty}(z; w \rmd\s)
    \]
    and observe that by \Cref{thm:AbsolutelyContinuousDeterminesEntropy} the right-hand side is equal to \(\l^{(d)}_{\infty}(z; \rmd\mu)\).

    On the other hand, since \(\Re( g(z) ) \geq 0\) for all \(z \in \BB^d\), we have that \(\Re( g_n(z) ) = \Re(g(z)) + \frac1n > 0\) for all \(z\in\BB^d\), so that \(g_n\) is nonzero on \(\BB^d\) and thus invertible in \(H^{\infty}(\BB^d)\), that is, for each \(n = 1, 2, \ldots\), we have that \(g_n\) satisfies the hypotheses of \Cref{thm:examplesinfty} and we conclude that
    \[
        \l_{\infty}^{(d)}(0; w_n \, \rmd\s) \leq \exp\left(\int_{\partial \BB^d} \log(w_n(\z)) \, \rmd \s(\z) \right).
    \]

    Since $w(\zeta) > 0$ ${\rm d}\sigma$-a.e., we have that $(\log(w_n))_{n=1}^{\infty}$ converges pointwise to $\log(w)$ and 
    \[
        \log(w_n) \leq w_n = \lvert g + 1/n \rvert^2 \leq  |g|^2 + 1 \leq \| g \|_{\infty} + 1
    \]
    on $\partial \BB^d$ ${\rm d}\sigma$-a.e. On the other hand, \(w(\z) \geq \e\) implies that
    \[
        \log(w_n) \geq \log\left(\e + \frac{1}{n}\right) \geq \log(\e) \quad \quad \text{${\rm d}\sigma$-a.e.};
    \]
    certainly the function \(\max\{\lvert \log(\e) \rvert, \norm{g}_{\infty} + 1\}\) is integrable on \(\partial\BB^d\). We may thus apply a well-known and straight-forward relaxation of the dominated convergence theorem \cite[Theorem 1.34]{Rud87} to obtain $\log w \in L^1(\partial \BB^d)$ and
    \[
        \lim_{n\to\infty} \int_{\partial\BB^d} \log(w_n(\z)) \rmd\s(\z) = \int_{\partial\BB^d} \log(w(\z)) \rmd\s(\z).
    \]

    Combining these relationships and using continuity of \(\exp(\cdot)\), we see that
    \[
        \l^{(d)}_{\infty}(z; \rmd \mu) \leq \lim_{n\to\infty}\exp\left(\int_{\partial \BB^d} \log(w_n(\z)) \, \rmd \s(\z) \right) = \exp\left(\int_{\partial \BB^d} \log(w(\z)) \, \rmd \s(\z) \right) > 0.
    \]
    The converse inequality was shown in \Cref{thm:BoundedAbove} and so the result follows via \Cref{viii=v}.
\end{proof}

\begin{remark}[Amplification of Theorem \ref{thm:examplesinfty} and Theorem \ref{thm:examplesinftyapprox}]
    \label{rem:FurtherExamples}
    One can quickly obtain additional related classes from the above results. Firstly, one could repeat the proof of \Cref{thm:examplesinftyapprox} with the assumption that \(\Im(g(z)) \geq 0 \) in place of the assumption that \(\Re(g(z)) \geq 0\), using instead \(g_n(z) = g(z) + \frac{i}{n}\), to obtain a symmetric class of examples. 

    Secondly, as we are interested primarily in \(w(\z) = \lvert g(\z) \rvert^2\), we could replace \(g\) by \(\overline{g}\) where \(g\) satisfies the assumptions of any of the three preceding results and essentially the same argument would allows us to again conclude the Szeg\H{o}--Verblunsky theorem.

    Finally, note that with \Cref{thm:BoundedAbove} in hand, it is sufficient for the preceding results of this section to prove the inequality \(\l_\infty^{(d)}(z; \rmd\mu) \leq \exp(\int_{\partial \BB^d} \log(w(\z)) \, \rmd \s(\z))\). As such, we could replace \(g\) by any \(h\) such that there exists \(g\) satisfying the hypotheses of the above results and satisfying \(\lvert h \rvert < \lvert g \rvert\) on \(\partial\BB^d\), and then once more the same arguments allow us to conclude the Szeg\H{o}--Verblunsky theorem for a non-trivial probability measure \(\rmd\mu(\z) = \lvert h(\z) \rvert^2 \rmd\s(\z) + \rmd\mu_\rms(\z)\).
\end{remark}

\begin{remark} 
    For a concrete example, any \emph{stable polynomial} --- that is, any \(q \in \CC[z_1, \ldots, z_d]\) such that \(\lvert q(z) \rvert > 0\) for all \(z \in \overline{\BB}^d\) --- with \(\norm{q}_{p} < 1\) for some \(p \in [1,\infty]\) gives rise to a \(w = |q|^p \in L^1(\partial\BB^d)\) such that the probability measure \(\rmd\mu(\z) = w(\z)\rmd\s(\z) + \rmd\mu_\rms(\z)\) with \(\mu_\rms\) discrete. Then the Verblunsky coefficients of $\mu$, say $(\gamma_{\a,\b})_{\a,\b \in \NN_0^d}$, will obey the formula
        $$\prod_{\a \in \NN_0^d} (1 - \lvert \g_{0,\a}\rvert^2) = \exp\left(\int_{\partial \BB^d} \log(w(\z)) \, \rmd \s(\z) \right).$$
\end{remark}

Our final result is a non-example showing that --- in stark contrast to the univariate and known polytorus settings --- the Szeg\H{o}--Verblunsky theorem does not always hold on \(\partial\BB^d\), i.e., the Szeg\H{o} entropy of a measure on \(\partial\BB^d\) does not always coincide with the product \(\prod_{\a \in \NN_0^d} (1 - \lvert \g_{0,\a} \rvert^2)\).

\begin{theorem}
    \label{thm:counterexample}
    There exists a measure \(\rmd\mu(\z) = \tilde{w}(\z) \rmd\s(\z)\) on \(\partial\BB^2\) such that
    \[
        \prod_{\a \in \NN_0^d} (1 - \lvert \g_{0,\a}\rvert^2) \neq \exp\left(\int_{\partial \BB^d} \log(\tilde{w}(\z)) \, \rmd \s(\z) \right).
    \]
\end{theorem}
\begin{proof}
    We consider a particular example that allows us to explicitly compute these quantities. Let \(d = 2\) (though the example generalises to \(d > 2\) in a natural fashion) and \(w(\z) = \lvert \z_1 \rvert^2 = 1 - \lvert \z_2 \rvert^2\) for \(\z \in \partial\BB^2\); we can integrate \(w\) to get
    \[
        \int_{\partial\BB^2} w(\z) \,\rmd\s(\z) = \int_{\partial\BB^2} \,\rmd\s - \int_{\partial\BB^2} \lvert \z_1 \rvert^2 \, \rmd\s(\z) = 1 - \frac{(2-1)!1!}{(2 - 1 + 1)!} = \frac{1}{2},
    \]
    by \cite[Proposition 1.4.9]{Rud08} with \(\a = (1, 0)\). Thus \(\rmd \mu(\z) = \tilde{w}(\z)\rmd\s(\z) :=  2 w(\z)\rmd\s(\z)\) is a probability measure on \(\partial\BB^2\) with discrete (indeed, zero) singular part. We may compute the moments of this measure: for \(\a,\b \in \NN_0^2\),
    \begin{align*}
        K(\a, \b) & = \int_{\partial\BB^2} \z^\a \overline{\z}^\b \,\rmd\mu(\z) = 2 \int_{\partial\BB^2} \z_1^{\a_1 + 1} \z_2^{\a_2} \overline{\z_1}^{\b_1 + 1} \overline{\z_2}^{\b_2} \, \rmd\s(\z);
    \end{align*}
    by \cite[Proposition 1.4.8]{Rud08}, this is zero when \((\a_1 + 1, \a_2) \neq (\b_1 + 1, \b_2)\), i.e. when \(\a \neq \b\). (One may also use \cite[Proposition 1.4.9]{Rud08} to compute a combinatorial expression for \(K(\a,\a)\) explicitly, but this shall turn out to be irrelevant.)

    As \(K(\a,\b) = 0\) if \(\a \neq \b\), \(K\) is diagonal. It follows that the moment matrix determinants given by \(D_\a = \det[K(\a',\b')]_{\a',\b' \preceq \a}\) are simply the products of the diagonal elements \(K(\b,\b)\) with \(\b\) ranging from \((0,0)\) to \(\a\):
    \[
        D_\a = \prod_{0 \preceq \b \preceq \a} K(\b, \b),
    \]
    and it follows that the ratios of these determinants are given by
    \[
        \frac{D_\a}{D_{\mathrm{prec}({\a})}} = K(\a,\a).
    \]
    
    Next, by \Cref{thm:MainResult}, (iii) = (iv), we have for \(\a \in \NN_0^2\) that
    \[
        \frac{D_{\a}}{D_{\mathrm{prec}(\a)}} = K(\a, \a) \cdot \prod_{0 \preceq \b \preceq \mathrm{prec}(\a)} (1 - \lvert \g_{\b, \a} \rvert^2),
    \]
    that is, for all \(\a \in \NN_0^2\), we have that
    \[
        K(\a,\a) = K(\a,\a) \cdot \prod_{0 \preceq \b \preceq \mathrm{prec}(\a)} (1 - \lvert \g_{\b, \a} \rvert^2)
    \]
    and hence
    \[
        \prod_{0 \preceq \b \preceq \mathrm{prec}(\a)} (1 - \lvert \g_{\b, \a} \rvert^2) = 1.
    \]
    As a finite product of real numbers satisfying \(0 \leq 1 - \lvert \g_{\b,\a} \rvert^2 \leq 1\), this product may only equal 1 when \(\g_{\b,\a} = 0\) for all \(0 \preceq \b \preceq \mathrm{prec}(\a)\). In particular, then, we have that \(\g_{0,\a} = 0\), and as \(\a \in \NN_0^d\) was arbitrary in the above argument, this holds for all \(\a \in \NN_0^2\).

    Thus item (v) of \Cref{thm:MainResult} becomes
    \[
        \prod_{\b = 0}^{\a(n)} (1 - \lvert \g_{0,\b}\rvert^2) = 1
    \]
    for any \(n \in \NN\), and hence by taking the limit we get
    \[
        \prod_{\a \in \NN_0^2} (1 - \lvert \g_{0,\a} \rvert^2) = 1.
    \]

    On the other hand, the Szeg\H{o} entropy of this measure is also computable, as follows. First, we remark that the ``equator'' of \(\partial\BB^2\) given by \(E = \{(0, \z_2) : \lvert \z_1 \rvert^2 = 1\}\) is of \(\rmd\s\)-measure zero, so
    \[
        \int_{\partial\BB^2} \log w(\z) \,\rmd\s(\z) = \int_{\partial\BB^2\setminus E} \log w(\z) \, \rmd\s(\z).
    \]
    Now, \(\lvert \z_2 \rvert^2 \neq 1\) on \(\partial\BB^2 \setminus E\), so we may use the Taylor series for \(\log(1 - x)\) to see that
    \[
        \log w(\z) = \log(1 - \lvert \z_2 \rvert^2) = - \sum_{n=1}^{\infty} \frac{1}{n} \lvert \z_2 \rvert^{2n}
    \]
    and hence (using compactness of \(\partial\BB^2\) to exchange the sum and integral and again using that \(E\) is a \(\rmd\s\)-null set)
    \begin{align*}
        \int_{\partial\BB^2} \log w(\z) \, \rmd\s(\z) & = \int_{\partial\BB^2 \setminus E} - \sum_{n=1}^{\infty} \frac{1}{n} \lvert \z_2 \rvert^{2n} \, \rmd\s(\z) \\
        & = - \sum_{n=1}^{\infty} \frac{1}{n} \int_{\partial\BB^2 \setminus E} \lvert \z_2 \rvert^{2n} \, \rmd\s(\z) \\
        & = - \sum_{n=1}^{\infty} \frac{1}{n} \int_{\partial\BB^2} \lvert \z_2 \rvert^{2n} \, \rmd\s(\z) \\
        & = - \sum_{n=1}^{\infty} \frac{1}{n} \int_{\partial\BB^2} \lvert \z^{(0,n)} \rvert^2 \, \rmd\s(\z) \\
        & = - \sum_{n=1}^{\infty} \frac{1}{n} \frac{(2-1)!0!n!}{(2 - 1 + 0 + n)!} \\
        & = - \sum_{n=1}^{\infty} \frac{1}{n} \frac{n!}{(n+1)!} \\
        & = - \sum_{n=1}^{\infty} \frac{1}{n(n+1)} \\
        & = - 1,
    \end{align*}
    again computing the integral via \cite[Proposition 1.4.9]{Rud08}.

    Recall that the Radon--Nikodym derivative of \(\mu\) is by definition \(\tilde{w} = 2w\). The Szeg\H{o} entropy of \(\mu\) is then
    \begin{align*}
        \exp\left(\int_{\partial \BB^2} \log(2w(\z)) \, \rmd \s(\z) \right) & = \exp\left(\int_{\partial\BB^2} \log w(\z) \,\rmd\s(\z) + \log(2)\right) \\
        & = 2\exp\left(\int_{\partial \BB^2} \log w(\z) \, \rmd \s(\z) \right) = 2\exp(-1) \\
        & = \frac2e.
    \end{align*}

    This does not coincide with 1, which we saw was the value of \(\prod_{\a \in \NN_0^2} (1 - \lvert \g_{0,\a} \rvert^2)\) and hence (as the singular part of \(\mu\) is discrete) items (v) - (ix) of \Cref{thm:MainResult}, that is, we have constructed a measure \(\rmd\mu(\z) = \tilde{w}(\z)\rmd\s(\z)\) on \(\partial\BB^2\) such that
    \[
        \prod_{\a \in \NN_0^d} (1 - \lvert \g_{0,\a}\rvert^2) \neq \exp\left(\int_{\partial \BB^2} \log(\tilde{w}(\z)) \, \rmd \s(\z) \right).
    \]
\end{proof}

\begin{remark}
    The function \(g(z_1, z_2) = \sqrt{2}z_1\) is a polynomial, and hence in \(H^p(\BB^2)\) for any \(p \in [1,\infty]\), and we have that the \(\tilde{w}\) of \Cref{thm:counterexample} is given by \(\tilde{w}(\z) = \lvert g(\z) \rvert^2\), in the same manner as in the earlier results of this section. However, \(g\) does not satisfy either the property \(\Re(g(z)) \geq 0\) on \(\BB^2\) nor that \(\Im(g(z)) \geq 0\) on \(\BB^2\) (nor the stronger property that \(1/g \in H^{\infty}(\BB^2)\)).
\end{remark}

\begin{remark}
    Certainly, the Szeg\H{o}--Verblunsky theorem shows that no such counterexample can exist when \(d = 1\). We also remark that, explicitly, the particular construction above breaks down when \(d=1\) as the function \(w(\z) = \lvert \z \rvert^2 \equiv 1\) is constant, so integrating \(w\) yields 1 and hence \(\tilde{w} = w\); then one may again show that
    \[
        \prod_{k \in \NN_0} (1 - \lvert \g_{0,k}\rvert^2) = 1,
    \]
    but this will now coincide with the Szeg\H{o} entropy
    \[
        \exp\left(\int_0^{2\pi} \log(\tilde{w}(e^{i\th})) \, \frac{\rmd\th}{2\pi} \right) = \exp(0) = 1.
    \]
\end{remark}

\bibliographystyle{plain}

\end{document}